\documentclass{cocv}
\newcommand{\R}{\mathbb{R}}

\newcommand{\mD}{\mathcal{D}}

\newcommand{\dd}{\, \text{d}}
\usepackage[ruled,linesnumbered]{algorithm2e}
\usepackage{xcolor}

\theoremstyle{plain}
\newtheorem{theorem}{Theorem}
\newtheorem{proposition}[theorem]{Proposition}
\newtheorem{lemma}[theorem]{Lemma}
\newtheorem{corollary}[theorem]{Corollary}

\theoremstyle{remark}
\newtheorem{remark}[theorem]{Remark}

\begin{document}
\title{The Effect of the Terminal Penalty in Receding Horizon Control for a Class of Stabilization Problems}\thanks{This project has received funding from the European Research Council (ERC) under the European Union’s Horizon 2020 research and innovation programme (grant agreement No 668998).}
\author{Karl Kunisch}\address{University of Graz, Austria and RICAM Institute, Linz, Austria.}
\author{Laurent Pfeiffer}\address{University of Graz, Austria.}
\date{November 2018}
\begin{abstract}
The Receding Horizon Control (RHC) strategy consists in replacing an infinite-horizon stabilization problem by a sequence of finite-horizon optimal control problems, which are numerically more tractable. The dynamic programming principle ensures that if the finite-horizon problems are formulated with the exact value function as a terminal penalty function, then the RHC method generates an optimal control. This article deals with the case where the terminal cost function is chosen as a cut-off Taylor approximation of the value function. The main result is an error rate estimate for the control generated by such a method, when compared with the optimal control. The obtained estimate is of the same order as the employed Taylor approximation and decreases at an exponential rate with respect to the prediction horizon. To illustrate the methodology, the article focuses on a class of bilinear optimal control problems in infinite-dimensional Hilbert spaces.
\end{abstract}
%
%
\maketitle

\section{Introduction}

In this article, we consider a bilinear optimal control problem of the following form:
\begin{equation}\label{eq:setup}
\begin{aligned}
& \inf_{u \in L^2(0,\infty)} \mathcal{J}(u,y_0) := \frac{1}{2}
\int_0^\infty \| C y(t) \|_Z^2 \dd t
+ \frac{\alpha}{2} \int_0^\infty u(t)^2 \dd t, \\
& \qquad \text{where: }
\left\{ \begin{array} {l} \displaystyle
\dot{y}(t)= Ay(t)+ (N y(t)+B) u(t), \quad
\text{for
} t>0 \\
y(0)=y_0 \in Y.
\end{array} \right.
\end{aligned}
\end{equation}
Here $V \subset Y \subset V^*$ is a Gelfand triple of real Hilbert spaces, where the embedding of $V$ into $Y$ is dense and compact, and $V^*$ denotes the topological dual of $V$. The operator
$A\colon \mD(A) \subset Y \rightarrow Y$ is the infinitesimal generator of an
analytic $C_0$-semigroup $e^{At}$ on $Y$, $B \in Y$, $C \in \mathcal{L}(Y,Z)$, $N \in \mathcal{L}(V,Y)$, $\alpha> 0$, and $\mD(A)$ denotes the domain of $A$.
The precise conditions on $A$, $B$, $C$, and $N$ are given further below. Under a detectability assumption on $(A,C)$, \eqref{eq:setup} is a stabilization problem, the goal being to steer $y$ to the origin, a steady state of the dynamical system (when $u=0$).
We denote by $\mathcal{V}$ the associated value function, i.e.\@ $\mathcal{V}(y_0)$ is the value of Problem \eqref{eq:setup} with initial condition $y_0$.

We aim at analyzing a specific receding horizon control (RHC) approach for solving \eqref{eq:setup}.
Before describing our contribution in some detail, let us briefly recapture some aspects of the receding horizon control methodology. In a nutshell, receding horizon control (also called model predictive control) consists in replacing the infinite-horizon control problem \eqref{eq:setup} by a sequence of finite-horizon problems with prediction horizon $T$.
At the beginning of  iteration $n$ of the method, a suboptimal control $u_{RH}$ and the associated trajectory $y_{RH}$ have been computed on the interval $(0,n\tau)$, where the sampling time $\tau >0$ is such that $\tau \leq T$.
The finite-horizon problem to be solved has the following form:
\begin{equation}\label{eq1.2}
\begin{aligned}
& \inf_{u \in L^2(0, T)} \frac{1}{2}
\int_{0}^{T} \| C y(t) \|_Z^2 \dd t
+ \frac{\alpha}{2} \int_{0}^{T} u(t)^2 \dd t + \phi(y(T)), \\
& \qquad \text{where: }
\left\{ \begin{array} {l} \displaystyle
\dot{y}(t)= Ay(t)+ (N y(t)+B) u(t), \quad
\text{for
} t \in (0,T) \\
y(0)= y_{RH}(n\tau) \in Y,
\end{array} \right.
\end{aligned}
\end{equation}
where $\phi$ denotes a terminal penalty function.
The control $u_{RH}$ and the trajectory $y_{RH}$ are extended on the interval $(n\tau,(n+1)\tau)$ by concatenation: $u_{RH}(n\tau +t)= u(t)$, $y_{RH}(n\tau + t)= y(t)$, for a.e. $t \in (0,\tau)$, where $u$ is a solution to \eqref{eq1.2} and $y$ is the associated trajectory.

The RHC method is receiving a tremendous amount of attention and it is frequently used in control engineering, in particular because problem \eqref{eq1.2} is easier to solve numerically than the infinite-horizon one. Another reason is that the method can be used as a feedback mechanism: the value of $u_{RH}$ on the interval $(n\tau,(n+1)\tau)$ is a function of $y_{RH}(n\tau)$, which implies that on that interval, the control mechanism can take into account possible perturbations having arisen before $n\tau$.

Let us review the main different choices which have been considered in the literature for the terminal penalty function $\phi$ involved in the finite-horizon problem, in the context of stabilization problems to a steady state.
Originally authors used to consider terminal state constraints. For instance, one can impose in \eqref{eq1.2} that $y(T)$ lies in a ball of small radius around the steady state, see e.g.\@ \cite{ABQRW, KG, K, MM, SMR}. In that case, $\phi$ is the indicator function of a neighborhood of the steady state.
As an alternative, terminal penalty functions called control Lyapunov functions have been used for guaranteeing the stability of the controlled system, see \cite{FK, IK1, JYH, PND} and the references cited there. It was observed later that for the stabilization of certain classes of dynamical systems, no terminal penalty function is necessary at all. This was proposed in \cite{JH} and further analyzed in e.g.\@ \cite{GMTT, RA}.
Let us point at some additional references from the large literature on receding horizon control. For finite-dimensional systems, we mention \cite{AZ, GP, MRRS}, for infinite-dimensional systems, we mention \cite{AK2, AK1, G}, and for discrete-time systems the articles \cite{GMTT, GR}.

The starting point of the present article is the following observation: If the value function associated with \eqref{eq:setup} is chosen as terminal penalty function in \eqref{eq1.2}, then, as a consequence of the dynamic programming principle, the control produced by the RHC method is optimal (for Problem \eqref{eq:setup}).
The question then arises how approximations to the value function of known order effect the approximation order of optimal receding horizon controls.
Taylor approximations are natural candidates for terminal penalty functions.
The Taylor approximation of order $k$ is denoted $\mathcal{V}_k(y_0)$ and it is of the form
\begin{equation*}
\mathcal{V}_k(y_0)= \sum_{j=2}^k \frac{1}{j!} \mathcal{T}_j(y_0,...,y_0),
\end{equation*}
where the mappings $\mathcal{T}_2$,$\mathcal{T}_3$,...,$\mathcal{T}_k$ are bounded multilinear forms of order $2$,$3$,...,$k$, respectively. The bilinear form $\mathcal{T}_2$ is given by $\mathcal{T}_2(y_0,y_0)= \langle y_0, \Pi y_0 \rangle$, where $\Pi \in \mathcal{L}(Y)$ is the unique nonnegative self-adjoint operator satisfying the following Riccati equation:
\begin{equation*}
\langle A^* \Pi y_1,y_2 \rangle + \langle \Pi A y_1,y_2 \rangle + \langle
Cy_1,Cy_2\rangle - \frac{1}{\alpha}  \langle B, \Pi y_1\rangle \langle B, \Pi y_2\rangle= 0, \quad  \text{for all } y_1,y_2\in\mathcal{D}(A).
\end{equation*}
Observe that $\mathcal{V}_2(y_0)= \frac{1}{2} \langle y_0, \Pi y_0 \rangle$ is the value function associated with \eqref{eq:setup} when $N=0$.
The other multilinear forms (of order 3 and more) are characterized as the unique solutions to generalized Lyapunov equations.
We refer to our article \cite{BreKP17a} for the derivation of these equations for an infinite-dimensional bilinear problem and to the survey \cite{KreAH13} for general finite-dimensional systems.
Taylor expansions have been mainly used in the literature for computing polynomial feedback laws. We refer to \cite{BreKP17c} for details concerning the practical computation of the Taylor expansions. In that paper, a Taylor expansion of order 5 is obtained for a control problem of the Fokker-Planck equation, with domains of dimension 1 and 2. We also refer to \cite{BreKP18a} for the suboptimality analysis of such feedback laws, in the context of infinite-dimensional bilinear problems.
In the context of RHC methods, the case of second-order Taylor approximations (for the terminal penalty function) has been often considered in the literature (see \cite{ABQRW,MRRS}).
To our knowledge, high-order Taylor expansions of the value function have only been used in the preprint \cite{Kre16}.

The aim of the article is to give a theoretical answer to the following question: Does a high-order approximation of the value function ensure that the RHC method generates a high-order approximation of the optimal control? We also investigate the effect of a large prediction horizon on the quality of the approximation.
Our main result is the following estimate:
\begin{equation} \label{eq:main_setup_1}
\| u_{RH} - \bar{u} \|_{L^2(0,\infty)} = \mathcal{O} \big( e^{-\lambda(T-\tau)-\lambda kT} \| y_0 \|_Y^k \big),
\end{equation}
where $\bar{u}$ is the solution to \eqref{eq:setup} with initial condition $y_0$ and $k \geq 2$ is the order of the Taylor approximation. The real number $\lambda$ is defined by $\lambda= - \sup_{\mu \in \sigma(A_\pi)} \text{Re}(\mu) > 0$ and $A_\pi= A- \frac{1}{\alpha} BB^* \Pi$.
Let us mention that our result is of local nature. For a given order $k$, the above estimate holds for values of $y_0$ in a neighborhood of $0$ and for a sampling time $\tau$ and a prediction horizon $T$ assumed to be sufficiently large. This local nature is mainly due to the fact that Taylor approximations are only valid in a neighborhood of the steady state.

In the last section of the article, we also consider the situation of quadratic terminal cost functions of the form: $\phi(y)= \frac{1}{2} \langle y, Qy \rangle_Y$, where $Q \in \mathcal{L}(Y)$ is symmetric and positive semi-definite. For this situation, we have the following estimate:
\begin{equation} \label{eq:main_setup_2}
\| u_{RH} - \bar{u} \|_{L^2(0,\infty)} = \mathcal{O} \big( e^{-\lambda(T-\tau)- \lambda T} \big( \| Q-\Pi \|_{\mathcal{L}(Y)} + e^{-\lambda T} \| y_0 \|_Y \big) \| y_0 \|_Y \big).
\end{equation}
Since $Q= 0$ is allowed, we cover the situation of a null terminal cost function.

The analysis is based on an estimation of the violation of the optimality conditions. More precisely, one can easily see that if $\phi= \mathcal{V}$ is replaced by $\phi= \mathcal{V}_k$ in Problem \eqref{eq1.2}, then only the terminal condition in the costate equation is modified in the corresponding optimality conditions. An error estimate for the control is then obtained by applying the inverse mapping theorem. This approach is quite common in the sensitivity analysis of optimization problems but it seems that it has never been applied before in the context of the RHC method.

The methodology which is presented in this article can be extended to other types of systems. In particular, in a finite-dimensional setting, the estimates \eqref{eq:main_setup_1} and \eqref{eq:main_setup_2} can be established if the non-linearity $Nyu$ is replaced by a more general term of the form $f(y,u)$, where $f$ is smooth and satisfies $f(0,0)= 0$ and $Df(0,0)=0$.
We have decided here to focus on bilinear systems, since they arise in the context of control of the Fokker-Planck equation, see \cite{BreKP17a,HarST13}. This also enables us to rely on some results obtained in \cite{BreKP18a}. The presented concepts can be applied to other nonlinear control systems, but they still require different adapted nonlinear PDE techniques.

The article is structured as follows.
In Section \ref{section:sectionFormulation} we state our main result. We also introduce the weighted spaces, which play an important role in our analysis. We recall in Section \ref{section:sensibility} some results concerning the dependence of the solution to \eqref{eq:setup} with respect to the initial condition $y_0$. Section \ref{section:finite-horizon} contains the core of our analysis. We estimate the violation of the optimality conditions and deduce an estimate for the solution to \eqref{eq1.2}. We finally obtain an estimate for the whole control generated by the RHC method in Section \ref{section:rhc_estimate}. The case of general quadratic cost functions is discussed in Section \ref{section:quad_tc}.

\section{Formulation of the problem, first properties, and main result} \label{section:sectionFormulation}

\subsection{Vector spaces}

For $T \in (0,\infty]$, we make use of the space $W(0,T) = \big\{ y \in L^2(0,T;V) \,|\, \dot{y} \in L^2 (0,T;V^*) \big\}$.
It is well-known that the space $W(0,T)$ is continuously embedded in $C_b([0,T],Y)$. We can therefore equip it with the following norm:
\begin{equation*}
\| y \|_{W(0,T)}
= \max \big( \| y \|_{L^2(0,T;V)}, \| \dot{y} \|_{L^2(0,T;V^*)}, \| y \|_{L^\infty(0,T;Y)} \big).
\end{equation*}

Let $\mu \in \R$ be given and let $T \in (0,\infty)$. Let us mention that the weighted spaces introduced here are only considered with a finite horizon $T$.
We denote by $L_{\mu}^2(0,T)$ the space of measurable functions $u \colon (0,T) \rightarrow \R$ such that
\begin{equation*}
\| u \|_{L_{\mu}^2(0,T)} := \| e^{\mu \cdot} u(\cdot) \|_{L^2(0,T)} =  \Big( \int_0^T \big( e^{\mu t} u(t) \big)^2 \dd t \Big)^{1/2} < \infty.
\end{equation*}
Observing that the mapping
$u \in L_{\mu}^2(0,T) \mapsto e^{\mu \cdot} u \in L^2(0,T)$
is an isometry, we deduce that $L_{\mu}^2(0,T)$ is a Banach space. Since $e^{\mu \cdot}$ is bounded from above and from below by a positive constant, we have that for all measurable $u \colon (0,T) \rightarrow \R$, $u \in L^2(0,T)$ if and only if $u \in L^2_\mu(0,T)$. The spaces $L^2(0,T)$ and $L^2_\mu(0,T)$ are therefore the same vector space, equipped with two different norms.
Similarly, we define the space $L_\mu^\infty(0,T;Y)$ of measurable mappings $y \colon (0,T) \rightarrow Y$ such that
\begin{equation*}
\| y \|_{L_\mu^\infty(0,T;Y)} := \| e^{\mu \cdot} y(\cdot) \|_{L^\infty(0,T;Y)} < \infty.
\end{equation*}
We finally define the Banach space $W_\mu(0,T)$ as the space of measurable mappings $y \colon (0,T) \rightarrow V$ such that $e^{\mu \cdot} y \in W(0,T)$. One can check that for all measurable mappings $y \colon (0,T) \rightarrow V$, $y \in W(0,T)$ if and only if $y \in W_\mu(0,T)$. The norm $\| \cdot \|_{W_{\mu}(0,T)}$ is defined by $\| y \|_{W_\mu(0,T)} = \| e^{\mu \cdot} y(\cdot) \|_{W_\mu(0,T)}$.

For $T \in (0,\infty)$ and $\mu \in \R$, we introduce the spaces
\begin{equation} \label{eq:spaceLambda}
\Lambda_{T,\mu}= W_{\mu}(0,T) \times L_{\mu}^2(0,T) \times W_{\mu}(0,T),
\end{equation}
that we equip with the norm
$\| (y,u,p) \|_{\Lambda_{T,\mu}} = \max \big( \| y \|_{W_\mu(0,T)}, \| u \|_{L_\mu^2(0,T)}, \| p \|_{W_\mu(0,T)} \big),
$
and
\begin{equation} \label{eq:spaceOmega}
\Upsilon_{T,\mu}= Y \times L_{\mu}^2(0,T;V^*) \times L_{\mu}^2(0,T;V^*) \times L_{\mu}^2(0,T) \times Y,
\end{equation}
that we equip with the following norm:
\begin{equation*}
\| (y_0,f,g,h,q) \|_{\Upsilon_{T,\mu}}
= \max \big( \| y_0 \|_Y, \| f \|_{L_{\mu}^2(0,T;V^*)}, \| g \|_{L_{\mu}^2(0,T;V^*)}, \| h \|_{L_{\mu}^2(0,T)}, e^{\mu T} \| q \|_{Y} \big).
\end{equation*}
Let us emphasize the fact that the component $q$ appears with a weight $e^{\mu T}$ in the above norm. The spaces $\Lambda_{T,0}$ and $\Lambda_{T,\mu}$ (resp.\@ $\Upsilon_{T,0}$ and $\Upsilon_{T,\mu}$) are the same vector space, equipped with two different norms. 
In the following lemma, the equivalence between these two norms is quantified (see \cite[Lemma 1.1]{BreP18} for a proof).

\begin{lemma} \label{lemma:embedding}
For all $\mu_0$ and $\mu_1$ with $\mu_0 \leq \mu_1$, there exists a constant $M>0$ such that for all $T \in (0,\infty)$ and for all $(y,u,p) \in \Lambda_{T,0}$,
\begin{align*}
\| (y,u,p) \|_{\Lambda_{T,\mu_0}} \leq \ & M \| (y,u,p) \|_{\Lambda_{T,\mu_1}}, \\
\| (y,u,p) \|_{\Lambda_{T,\mu_1}} \leq \ & M e^{(\mu_1-\mu_0)T} \| (y,u,p) \|_{\Lambda_{T,\mu_0}},
\end{align*}
and such that, similarly, for all $(y_0,f,g,h,q) \in \Upsilon_{T,0}$,
\begin{align*}
\| (y_0,f,g,h,q) \|_{\Upsilon_{T,\mu_0}} \leq \ & M \| (y_0,f,g,h,q) \|_{\Upsilon_{T,\mu_1}}, \\
\| (y_0,f,g,h,q) \|_{\Upsilon_{T,\mu_1}} \leq \ & M e^{(\mu_1-\mu_0)T} \| (y_0,f,g,h,q) \|_{\Upsilon_{T,\mu_0}}.
\end{align*}
\end{lemma}

Finally, we make occasionally use of the spaces $\Lambda_{\infty,0}:= W(0,\infty) \times L^2(0,\infty) \times W(0,\infty)$ and $\Upsilon_{\infty,0}:= Y \times L^2(0,\infty;V^*) \times L^2(0,\infty;V^*) \times L^2(0,\infty)$, equipped with the following norms:
\begin{align*}
\| (y,u,p) \|_{\Lambda_{\infty,0}} = \ & \max \big( \| y \|_{W(0,\infty)}, \| u \|_{L^2(0,\infty)}, \| p \|_{W(0,\infty)} \big) \\
\| (y_0,f,g,h) \|_{\Upsilon_{\infty,0}}
= \ & \max \big( \| y_0 \|_Y, \| f \|_{L^2(0,\infty;V^*)}, \| g \|_{L^2(0,\infty;V^*)}, \| h \|_{L^2(0,\infty)} \big).
\end{align*}
Note that the elements of $\Upsilon_{\infty,0}$ do not have a component $q$, to the contrary of those in $\Upsilon_{T,0}$.

\subsection{Assumptions}

Throughout the article we assume that the following four assumptions hold true.
\begin{itemize}
\item[(A1)] The operator $-A$ can be associated with a $V$-$Y$ coercive bilinear
form $a\colon V\times V \to \mathbb R$ such that there exist $\lambda_0 \in \mathbb
R$ and $\delta > 0$ satisfying $a(v,v) \geq \delta \| v \| _V^2 - \lambda_0 \| v\|_Y^2$, for all $v \in V$.
\item[(A2)] The operator $N$ is such that $N \in \mathcal{L}(V,Y)$ and
$N^* \in \mathcal{L}(V,Y)$.
\item[(A3)] [Stabilizability] There exists an operator $F \in \mathcal{L}(Y,\R)$ such that the semigroup $e^{(A+BF) t}$ is exponentially stable on $Y$.
\item[(A4)] [Detectability] There exists an operator $K \in \mathcal{L}(Z,Y)$ such that the semigroup $e^{(A-KC) t}$ is exponentially stable on $Y$.
\end{itemize}

Let us mention that a simple example of stabilisation problem satisfying these assumptions is given in \cite[Example 2.3]{BreKP18a}. The assumptions are also satisfied for a class of control problems of the Fokker-Planck equation (see the discussion in \cite[Section 8]{BreKP17b}).
Assumptions (A3) and (A4) are well-known and analysed in infinite-dimensional systems theory, see \cite{CurZ95}, for example. In particular, there has been ongoing interest on stabilizability of infinite-dimensional parabolic systems by finite-dimensional controllers. We refer to \cite{BT11, Ray18} and the references given there. Assumptions (A3) and (A4) play an important role all along the article.
While the results of this article are obtained for scalar controls, the generalization to the case of systems of the form $\dot{y}= Ay + \sum_{j=1}^m(N_j y(t)+B_j) u_j(t)$,
with $B_j \in Y$, can easily be achieved. Assumption (A3) must be replaced by the following one: there exist operators $F_1$,...,$F_m$ in $\mathcal{L}(Y,\R)$ such that the semigroup $e^{(A + \sum_{j=1}^m B_jF_j)t}$ is exponentially stable.

Consider now the algebraic operator Riccati equation:
\begin{equation} \label{eq:are}
\langle A^* \Pi y_1,y_2 \rangle + \langle \Pi A y_1,y_2 \rangle + \langle
Cy_1,Cy_2\rangle - \frac{1}{\alpha}  \langle B, \Pi y_1\rangle \langle B, \Pi y_2\rangle= 0  \text{ for all } y_1,y_2\in\mathcal{D}(A).
\end{equation}
Due to the (exponential) stabilizability and detectability assumptions, it is well-known (see \cite[Theorem 6.2.7]{CurZ95}) that \eqref{eq:are} has a unique nonnegative self-adjoint solution $\Pi \in \mathcal{L}(Y)$ and that additionally, the semigroup generated by the operator $A_\pi:=A-\frac{1}{\alpha}BB^*\Pi$ is exponentially stable on $Y$. Let us now fix
\begin{align}\label{eq:spect_absc}
\lambda  = -\sup\limits_{\mu \in \sigma(A_\pi)}( \mathrm{Re}(\mu)) > 0.
\end{align}
The constant $\lambda$ is the one involved in \eqref{eq:main_setup_1} and \eqref{eq:main_setup_2}.
The positivity of $\lambda$ is a consequence of the exponential stability of the semigroup generated by $A_\pi$. Let us mention that its exponential stability is a crucial property for the proof of Proposition \ref{proposition:non_reg_os}, given in the article \cite{BreP18}.

\subsection{Formulation of the problem}

We are now prepared to state the problem under consideration. For $y_0 \in Y$, consider
\begin{equation} \label{eqProblem} \tag{$P$}
	\inf_{u \in L^2(0,\infty)} \mathcal{J}(u,y_0):=  \frac{1}{2}
\int_0^\infty
	\| C S(u,y_0;t)\|^2_Z \dd t + \frac{\alpha}{2}
\int_0^\infty u(t)^2 \dd t,
\end{equation}
where $S(u,y_0;\cdot)$ is the solution to
\begin{equation}\label{eqStateEquation}
\left\{	\begin{array} {l}
\dot{y}(t)=Ay(t)+Ny(t)u(t)+Bu(t), \quad \text{for} \;t>0,\\
y(0)=y_0.
\end{array} \right.
\end{equation}
Here $y=S(u,y_0)$ is referred to as solution of \eqref{eqStateEquation} if for all $T>0$, it
lies in $W(0,T)$. The well-posedness of the state equation is ensured by Lemma \ref{lemma:RegEstim} below.
The lemma is a simple generalization of \cite[Lemma 1]{BreKP17b}.

\begin{lemma} \label{lemma:RegEstim}
For all $T > 0$ and $u \in L^2(0,T)$, there exists a unique solution $y \in W(0,T)$ to the following system:
\begin{equation*}
\dot{y}= Ay + Nyu + Bu, \quad y(0)= y_0.
\end{equation*}
Moreover, there exists a continuous function $c$ such that
$\| y \|_{W(0,T)} \leq c(T,\|y_0\|_Y, \|u\|_{L^2(0,T)})$.
\end{lemma}

Finally, we denote by $\mathcal{V} \colon Y \rightarrow [0,\infty]$ the value function associated with Problem \eqref{eqProblem}, defined by
\begin{equation*}
\mathcal{V} (y_0) = \inf_{u \in L^2(0,\infty)} \mathcal{J}(u,y_0).
\end{equation*}
Note that the origin is a steady state of the uncontrolled system \eqref{eqStateEquation} and that $\mathcal{V}(0)=0$.

\subsection{Main result}

The goal of this article is to analyze the efficiency of the RHC method when Taylor approximations of the value function are used as terminal cost functions. The following theorem, taken from \cite{BreKP18a}, states that the value function is locally infinitely many times differentiable.

\begin{theorem}[Theorem 6.6, \cite{BreKP18a}]
The value function $\mathcal{V}$ is real-valued and infinitely differentiable in a neighborhood of 0. Moreover, $D\mathcal{V}(0)= 0$, $D^2 \mathcal{V}(0)$ is the bilinear form associated with $\Pi$ (the solution to the algebraic Riccati equation \eqref{eq:are}) and for all all $k \geq 3$, $D^k \mathcal{V}(0)$ can be obtained as the unique solution to a generalized Lyapunov equation.
\end{theorem}

We denote by $\mathcal{V}_k \colon Y \rightarrow \R$ the Taylor expansion of order $k \geq 2$ of the value function around 0:
\begin{equation*}
\mathcal{V}_k(y)= \sum_{j=2}^k \frac{1}{j!} D^j \mathcal{V}(0)(y,...,y).
\end{equation*}
In the above expression, $D^j \mathcal{V}(0)$ is a bounded multilinear form from $Y^k$ to $\R$.
As explained in the introduction, the RHC method consists in solving a sequence of finite-horizon problems. The finite-horizon problems considered in the present article are as follows:
\begin{equation*} \label{eq:Pb_One_step_RH} \tag{$P_{T,k}$}
\begin{cases} \begin{array}{l}
{\displaystyle \inf_{(y,u) \in W(0,T) \times L^2(0,T)} \
\frac{1}{2} \int_0^T \| Cy(t) \|_{Z}^2 \dd t + \frac{\alpha}{2} \int_0^T u(t)^2 \dd t + \mathcal{V}_k(y(T)),} \\[1em]
\qquad \text{subject to: } \dot{y}= Ay + Nyu + Bu, \quad y(0)= y_0.
\end{array}
\end{cases}
\end{equation*}
Algorithm \ref{algo:rh} below describes the Receding-Horizon method.
\begin{algorithm}
Input: $\tau \geq 0$, $T \geq \tau$, $y_0 \in Y$\;
Set $n=0$ and $y_n= y_0$\;
\For{n=0,1,2,...}{
Find a local solution $(y_{T,k},u_{T,k})$ to Problem \eqref{eq:Pb_One_step_RH}, with initial condition $y_n$\;
For a.e. $t \in (0,\tau)$, define $y_{RH}(n \tau + t)= y_{T,k}(t)$ and $u_{RH}(n \tau + t)= u_{T,k}(t)$\;
Set $y_{n+1}= y_{RH}((n+1)\tau)$.
}
\caption{Receding-Horizon method}
\label{algo:rh}
\end{algorithm}

We next state the main result of this paper. It involves the solution to problem \eqref{eqProblem}, whose existence and uniqueness will be established in Proposition \ref{prop:existenceSol} below, as well as the local solutions  to  the auxiliary problems \eqref{eq:Pb_One_step_RH} which arise in the iterative steps of the receding-horizon control method. Let us recall that the constant $\lambda$ involved in the main result has been fixed in \eqref{eq:spect_absc}. We also denote by $B_Y(\delta)$ the closed ball of $Y$ of center 0 and radius $\delta$.

\begin{theorem} \label{thm:main}
For all $k \geq 2$, there exist $\tau_0 > 0$,  $\delta > 0$, and $M > 0$ such that for all $\tau \geq \tau_0$, for all $T \geq \tau$, and all $y_0 \in B_Y(\delta)$, the Receding-Horizon method is well-posed, assuming that the local solution to \eqref{eq:Pb_One_step_RH} obtained at each iteration is the one characterized in Proposition \ref{prop:estimate_one_step}. Moreover, the following estimates hold true:
\begin{align}
\max( \| y_{RH} - \bar{y} \|_{W(0,\infty)}, \| u_{RH} - \bar{u} \|_{L^2(0,\infty)} )
\leq \ & M e^{-\lambda(T-\tau)- \lambda kT} \| y_0 \|_Y^k \label{eq:final_estim1} \\
\mathcal{J}(u_{RH},y_0) - \mathcal{V}(y_0) \leq \ & M e^{-2\lambda(T-\tau)- 2\lambda kT} \| y_0 \|_Y^{2k}, \label{eq:final_estim2}
\end{align}
where $\bar{u}$ is the unique solution to problem \eqref{eqProblem} and $\bar{y}$ the associated trajectory.
\end{theorem}

The proof of the theorem is given in Section \ref{section:rhc_estimate}.

\begin{remark} \label{rem:th1}
The estimate \eqref{eq:final_estim1} is of order $k$ with respect to $\| y_0 \|_Y$. This is related to the fact that $D \mathcal{V}(y_0) = D\mathcal{V}_k(y_0) + O(\| y_0 \|_Y^k)$, as will be seen later.
Estimate \eqref{eq:final_estim1} suggests that the quality of the RHC control can be improved by increasing $T$ or reducing $\tau$. Still, the value of $\tau_0$ cannot be made arbitrary small, thus our estimate does not capture the behaviour of the RHC method for very small sampling times.
\end{remark}

\subsection{Linear optimality systems} \label{subsection:lin_opt_sys}

As was noticed in the introduction, the pairs $(y_{T,k},u_{T,k})$ and $(\bar{y}_{|(0,T)}, \bar{u}_{|(0,T)})$ satisfy similar optimality conditions: The only difference occurs in the terminal condition for the costate equation. A key issue for the proof of our main result is therefore the following: What is the impact of a modification of the terminal condition on the solution to the optimal control problem \eqref{eq:Pb_One_step_RH}? This is a typical issue of sensitivity analysis, which can be tackled with the inverse mapping theorem. In a nutshell, the inverse mapping theorem allows to prove that a certain mapping, ``containing" the first-order optimality conditions, is (locally) bijective. In order to apply the inverse mapping theorem, one needs to prove that the derivative of the mentioned mapping is bijective, which will be done several times in Section \ref{section:finite-horizon} with the help of the following proposition, which is demonstrated in \cite[Theorem 2.1]{BreP18}.

\begin{proposition} \label{proposition:non_reg_os}
Let $\mu \in \{ - \lambda, 0, \lambda \}$. Let $\mathcal{Q} \subset \mathcal{L}(Y)$ be a bounded set of symmetric positive semi-definite operators. For all $T>0$, $Q \in \mathcal{Q}$, and $(y_0,f,g,h,q) \in \Upsilon_{T,\mu}$, there exists a unique solution $(y,u,p) \in \Lambda_{T,\mu}$ to the following optimality system:
\begin{equation} \label{eq:non_reg_os}
\begin{cases}
\begin{array}{rll}
y(0) = & \! \! \! y_0 \qquad & \text{in $Y$} \\
\dot{y}-(Ay + Bu) = & \! \! \! f \qquad & \text{in $L_\mu^2(0,T;V^*)$} \\
-\dot{p} - A^* p - C^* C y = & \! \! \! g & \text{in $L_\mu^2(0,T;V^*)$} \\
\alpha u + \langle B, p \rangle_{V^*,V} = & \! \! \! - h & \text{in $L_\mu^2(0,T)$} \\
p(T) - Q y(T)= & \! \! \! q & \text{in $Y$}.
\end{array}
\end{cases}
\end{equation}
Moreover, there exists a constant $M$ independent of $T$, $Q$, and $(y_0,f,g,h,q)$ such that
\begin{equation} \label{eq:estimate_non_reg_os}
\| (y,u,p) \|_{\Lambda_{T,\mu}}
\leq M \| (y_0,f,g,h,q) \|_{\Upsilon_{T,\mu}}.
\end{equation}
\end{proposition}

\section{Sensitivity analysis for the non-linear problem} \label{section:sensibility}

In this section we gather some results from \cite{BreKP18a}.
The following proposition deals with the existence of a solution to \eqref{eqProblem} and with first-order necessary optimality conditions. All along the paper, the constants $M$ which are used are generic constants, whose value may change.

\begin{proposition}[Lemma 4.7, Proposition 4.8, \cite{BreKP18a}] \label{prop:basic_oc}
There exists $\delta_1>0$ such that for all $y_0 \in B_Y(\delta_1)$, Problem \eqref{eqProblem} with initial condition $y_0$ has a unique solution $u$. Moreover, there exists a unique costate $p \in W(0,\infty)$ such that
\begin{align*}
-\dot{p} - (A + uN)^* p - C^*Cy = \ & 0, \\
\alpha u + \langle p, Ny + B \rangle_Y = \ & 0,
\end{align*}
where $y= S(u,y_0)$.
\end{proposition}

Consider the mapping $\Phi_1$, defined as follows:
\begin{equation} \label{eq:defPhi1}
\Phi_1 \colon
(y,u,p) \in \Lambda_{\infty,0} \mapsto
\begin{pmatrix}
y(0) \\ \dot{y}- (Ay + (Ny+B)u) \\ - \dot{p} - A^*p - u N^* p - C^*C y \\
\alpha u + \langle Ny+ B,p \rangle_Y
\end{pmatrix}
\in \Upsilon_{\infty,0}.
\end{equation}
The mapping $\Phi_1$ is such that for all $(y,u,p) \in \Lambda_{\infty,0}$, the triplet $(y,u,p)$ satisfies the optimality conditions of Proposition \ref{prop:basic_oc} if and only if $\Phi_1(y,u,p)= (y_0,0,0,0)$.
The following proposition is a refinement of Proposition \ref{prop:basic_oc}.

\begin{proposition}[Lemma 4.7, Proposition 4.8, \cite{BreKP18a}] \label{prop:existenceSol}
There exist $\delta_1 > 0$, $\delta_1'>0$, $M>0$, and three $M$-Lipschitz continuous mappings
\begin{equation*}
y_0 \in B_Y(\delta_1) \mapsto (\mathcal{Y}_1(y_0),\mathcal{U}_1(y_0),\mathcal{P}_1(y_0)) \in \Lambda_{\infty,0}
\end{equation*}
such that the following holds:
\begin{enumerate}
\item For all $y_0 \in B_Y(\delta_1)$, $(\mathcal{Y}_1(y_0),\mathcal{U}_1(y_0),\mathcal{P}_1(y_0))$ is the unique solution to
\begin{equation} \label{eq:nonLinPhi1}
\Phi_1(y,u,p)= (y_0,0,0,0), \quad
\| (y,u,p) \|_{\Lambda_{\infty,0}} \leq \delta_1'.
\end{equation}
\item For all $y_0 \in B_Y(\delta_1)$, the control $\mathcal{U}_1(y_0)$ is the unique solution to \eqref{eqProblem} with initial condition $y_0$, with associated trajectory  $\mathcal{Y}_1(y_0)$ and costate $\mathcal{P}_1(y_0)$.
\end{enumerate}
\end{proposition}

\begin{proof}
The first part of the result is a direct consequence of the inverse mapping theorem (see Theorem \ref{thmInvThm} in the Appendix). We have $\Phi_1(0,0,0)=(0,0,0,0)$.
One can check that the mapping $\Phi_1$ is well-defined and continuously differentiable and that $D\Phi_1$ is globally Lipschitz continuous, since it only contains linear terms and three bilinear terms, $Nyu$, $uN^*p$, and $\langle Ny,p \rangle_Y$.
For all $(y,u,p) \in \Lambda_{\infty,0}$, for all $(w_1,w_2,w_3,w_4) \in \Upsilon_{\infty,0}$,
\begin{equation*}
D\Phi_1(0,0,0)(y,u,p)= (w_1,w_2,w_3,w_4)
\Longleftrightarrow
\begin{cases}
\begin{array}{rl}
y(0)= & w_1 \\
\dot{y} - (Ay + Bu)= & w_2 \\
-\dot{p} - A^*p - C^*Cy = & w_3 \\
\alpha u + \langle B, p \rangle_Y= & w_4.
\end{array}
\end{cases}
\end{equation*}
The above linear system has a unique solution $(y,u,p)$, moreover
$\| (y,u,p) \|_{\Lambda_{\infty,0}} \leq M \| (w_1,w_2,w_3,w_4) \|_{\Upsilon_{\infty,0}}$,
for some constant $M$ independent of $(w_1,w_2,w_3,w_4)$. We refer the reader to \cite[Lemmas 4.4 and 4.7]{BreKP18a} for a proof of existence and uniqueness and for the a priori bound.
This proves that $D\Phi_1(0,0,0)$ is invertible with a bounded inverse and finally, that the inverse mapping theorem applies.

For the second part of the theorem (the optimality of $\mathcal{U}_1(y_0)$), we refer to \cite[Lemma 4.7, Proposition 4.8]{BreKP18a}.
\end{proof}

In the sequel, we will write $(\mathcal{Y}_1,\mathcal{U}_1,\mathcal{P}_1)(y_0)$ instead of $(\mathcal{Y}_1(y_0),\mathcal{U}_1(y_0),\mathcal{P}_1(y_0))$.
Note that the mappings $\mathcal{Y}_1$, $\mathcal{U}_1$, and $\mathcal{P}_1$ will be used all along the article to indicate the solution to \eqref{eqProblem} and its associated trajectory and costate. Note also that $(\mathcal{Y}_1,\mathcal{U}_1,\mathcal{P}_1)(0)= (0,0,0)$. From time to time, we simply denote this triple by $(\bar{y},\bar{u},\bar{p})$, when the initial condition has been specified and no risk of confusion is possible.

Finally, we will also make use of the following result, known in the literature as sensitivity relation.

\begin{lemma}[Lemma 5.1, \cite{BreKP18a}] \label{lemma:sensitivity_relation}
There exists $\delta_2 \in (0,\delta_1]$ such that for all $y_0 \in B_Y(\delta_2)$, for all $t \in [0,\infty)$, the value function is differentiable at $y(t)$ with $p(t)= D \mathcal{V}(y(t))$, where $y= \mathcal{Y}(y_0)$ and $p= \mathcal{P}(y_0)$.
\end{lemma}

\section{Analysis of the finite-horizon problem} \label{section:finite-horizon}

From now on, the order of approximation $k$ of the Taylor expansion is fixed.
We start this section with a result concerning the existence of a solution to Problem \eqref{eq:Pb_One_step_RH} (Proposition \ref{prop:existence_one_step_pb}) and provide then optimality conditions (Lemma \ref{lemma:oc_for_one_step}). The comparison of the pairs $(y_{T,k},u_{T,k})$ and $(\bar{y}_{|(0,T)}, \bar{u}_{|(0,T)})$ (announced in subsection \ref{subsection:lin_opt_sys}) is done in Proposition \ref{prop:estimate_one_step}.

\begin{lemma} \label{lemma:observer}
There exists $\delta_3 >0$ and $M > 0$ such that for all $T \in (0,\infty)$, for all $u \in L^2(0,T)$ with $\| u \|_{L^2(0,T)} \leq \delta_3$, and for all $y_0 \in Y$, the following estimate holds:
\begin{equation*}
\| y \|_{W(0,T)} \leq M \big( \| y_0 \|_Y + \| u \|_{L^2(0,T)} + \| Cy \|_{L^2(0,T;Z)} \big),
\end{equation*}
where $y$ denotes the solution to the system: $\dot{y}= Ay + N y u + Bu$, $y(0)= y_0$.
\end{lemma}

A proof can be found in \cite[Lemma 2.7]{BreKP18a} for the case $T= \infty$. The proof can be directly adapted to the case of finite horizons.
The next proposition addresses  the existence of a local solution for Problem \eqref{eq:Pb_One_step_RH}, assuming that $\| y_0 \|_Y$ is sufficiently small.

\begin{proposition} \label{prop:existence_one_step_pb}
There exist $\delta_4 > 0$ and $M>0$ such that for all $y_0 \in B_Y(\delta_4)$, Problem \eqref{eq:Pb_One_step_RH} has a local solution $(y_{T,k},u_{T,k})$ satisfying
\begin{equation} \label{eq:bound_sol_one_step}
\max \big( \| y_{T,k} \|_{W(0,T)}, \| u_{T,k} \|_{L^2(0,T)} \big) \leq M \| y_0 \|_Y.
\end{equation}
If $k=2$, then Problem \eqref{eq:Pb_One_step_RH} has a global solution satisfying the above estimate.
\end{proposition}

\begin{proof}
Let us start with the case $k \geq 3$.
If $y_0= 0$, one can easily check that $(y_{T,k},u_{T,k})= (0,0)$ is a local solution to the problem.
From now on, we assume that $y_0 \neq 0$. Let us emphasize the fact that the constants $M_1$,...,$M_5$ introduced in this proof can all be chosen independently of $T$. The value of $\delta_4$ will be reduced along the proof, this can be done independently of $T$.

As a consequence of Proposition \ref{prop:existenceSol}, there exist $\delta_4 >0$ and $M_1$ such that for all $y_0 \in B_Y(\delta_4)$, Problem \eqref{eqProblem} with initial condition $y_0$ has a solution $\bar{u}$ with associated trajectory $\bar{y}$ satisfying
\begin{equation} \label{eq:bound_on_cost}
\frac{1}{2} \| C \bar{y} \|_{L^2(0,\infty;Z)}^2
+ \frac{\alpha}{2} \| \bar{u} \|_{L^2(0,\infty)}^2 \leq M_1 \| y_0 \|_Y^2, \quad
\| \bar{y} \|_{L^\infty(0,\infty;Y)} \leq M_1 \|y_0 \|_Y.
\end{equation}
We need to bound $\mathcal{V}_k$ from below. Observe that $\mathcal{V}_k$ need not be nonnegative.  Since it is a Taylor approximation of order 3 (at least), there exists a constant $M_2>0$ such that $|\mathcal{V}_k(y)-\mathcal{V}(y) | \leq M_2 \| y \|_Y^4$ for all $y \in B_Y(M_1 \delta_4)$, after possible reduction of $\delta_4$. Moreover, the value function $\mathcal{V}$ is non-negative, therefore
\begin{equation*}
\mathcal{V}_k(y) \geq \mathcal{V}(y) - |\mathcal{V}_k(y)-\mathcal{V}(y)| \geq - M_2 \| y \|_Y^4.
\end{equation*}
Increasing if necessary the value of $M_2$, we also have for all $y \in B_Y(M_1 \delta_4)$ the following upper estimate $\mathcal{V}_k(y) \leq M_2 \| y \|_Y^2$,
since $\mathcal{V}_k$ only contains terms of order 2 and more.

For a given $\gamma>0$, consider the following localized problem:
\begin{equation*} \label{eq:Pb_One_step_RH_loc} \tag{$P_{T,k,\gamma}$}
\begin{cases} \begin{array}{l}
{\displaystyle \inf_{(y,u) \in W(0,T) \times L^2(0,T)} \
J_{T,k}(y,u):= \frac{1}{2} \int_0^T \| Cy(t) \|_{Z}^2 \dd t + \frac{\alpha}{2} \int_0^T u(t)^2 \dd t + \mathcal{V}_k(y(T)),} \\[1em]
\qquad \text{subject to: } \dot{y}= Ay + Nyu + Bu, \quad y(0)= y_0, \\
\qquad \text{\phantom{subject to: }} \| y(T) \|_Y \leq \gamma \| y_0 \|_Y.
\end{array}
\end{cases}
\end{equation*}
Problem \eqref{eq:Pb_One_step_RH_loc} is similar to \eqref{eq:Pb_One_step_RH}, with the additional constraint: $\| y(T) \|_Y \leq \gamma \| y_0 \|_Y$. Our strategy now is the following: we prove the existence of a solution to \eqref{eq:Pb_One_step_RH_loc} such that the additional constraint is not active for an appropriately chosen value of $\gamma$. The obtained solution is then necessarily a local solution to \eqref{eq:Pb_One_step_RH}.

Let $\gamma \geq M_1$. For all $y_0 \in B_Y(\delta_4)$, the restriction to $(0,T)$ of the pair $(\bar{y},\bar{u})$ is feasible (for Problem \eqref{eq:Pb_One_step_RH_loc}), by \eqref{eq:bound_on_cost}.
Moreover, $\mathcal{V}_k(\bar{y}(T)) \leq M_2 \| \bar{y}(T) \|_Y^2 \leq M_1^2 M_2 \| y_0 \|_Y^2$, therefore
\begin{equation*}
J_{T,k}(\bar{y}_{|(0,T)},\bar{u}_{|(0,T)}) \leq (M_1 + M_1^2 M_2) \| y_0 \|_Y^2.
\end{equation*}
Consider now a minimizing sequence $(y_n,u_n)_{n \in \mathbb{N}}$ for \eqref{eq:Pb_One_step_RH_loc}. We can assume that for all $n \in \mathbb{N}$,
\begin{equation*}
J_{T,k}(y_n,u_n) \leq (M_1 + M_1^2 M_2) \| y_0 \|_Y^2.
\end{equation*}
Using the lower bound of $\mathcal{V}_k$, we obtain that for all $n \in \mathbb{N}$,
\begin{equation*}
\frac{1}{2} \| Cy_n \|_{L^2(0,T;Z)}^2 + \frac{\alpha}{2} \| u_n \|_{L^2(0,T)}^2
\leq (M_1 + M_1^2 M_2) \| y_0 \|_Y^2 + M_2 \gamma^4 \| y_0 \|_Y^4.
\end{equation*}
Therefore, there exists a constant $M_3>0$, independent of $T$ and $\gamma$, such that for all $n \in \mathbb{N}$,
\begin{equation*}
\| u_n \|_{L^2(0,T)} \leq M_3 \big( \|y_0 \|_Y + \gamma^2 \| y_0 \|_Y^2 \big)
\leq M_3 \big( \delta_4 + \gamma^2 \delta_4^2 \big).
\end{equation*}
Let us reduce the value of $\delta_4$, if necessary, so that $M_3 \big( \delta_4 + \gamma^2 \delta_4^2 \big) \leq \delta_3$.
Thus, for all $n \in \mathbb{N}$, $\| u_n \|_{L^2(0,T)} \leq \delta_3$. Applying Lemma \ref{lemma:observer}, we obtain that there exists a constant $M_4$, independent of $T$ and $\gamma$, such that for all $n \in \mathbb{N}$,
\begin{equation*}
\| y_n \|_{W(0,T)} \leq M_4 \big( \|y_0 \|_Y + \gamma^2 \| y_0 \|_Y^2 \big).
\end{equation*}
Thus, the sequence $(y_n,u_n)_{n \in \mathbb{N}}$ is bounded in $W(0,T) \times L^2(0,T)$. Using the techniques of \cite[Proposition 2]{BreKP17b}, one can show that all limit points $(y,u)$ of the sequence (there exists at least one) are solutions to Problem \eqref{eq:Pb_One_step_RH_loc} and satisfy:
\begin{equation} \label{eq:bound_uy}
\begin{cases}
\begin{array}{l}
\| u \|_{L^2(0,T)} \leq M_3 \big( \|y_0 \|_Y + \gamma^2 \| y_0 \|_Y^2 \big), \\
\| y \|_{L^2(0,T;V)} \leq M_4 \big( \|y_0 \|_Y + \gamma^2 \| y_0 \|_Y^2 \big), \\
\| \dot{y} \|_{L^2(0,T;V^*)} \leq M_4 \big( \|y_0 \|_Y + \gamma^2 \| y_0 \|_Y^2 \big).
\end{array}
\end{cases}
\end{equation}
We need to find an estimate on $\| y \|_{L^\infty(0,T;Y)}$. As usual, this is achieved by multiplying the state equation by $y$, estimating the right-hand side with Young's inequality and then applying Gronwall's lemma. Following the first steps of the proof of \cite[Lemma 1]{BreKP17b}, we obtain the existence of a constant $M> 0$ (independent of $t$ and $T$) such that
\begin{equation*}
\frac{\dd}{\dd t} \| y(t) \|_Y^2
\leq M \big( \| y(t) \|_Y^2 + |u(t)|^2 + \| y \|_Y^2 |u(t)|^2  \big), \quad \forall t \in [0,T].
\end{equation*}
Applying Gronwall's lemma, we obtain that
\begin{equation*}
\| y(t) \|_Y^2 \leq \Big( \| y(0) \|_Y^2 + \int_0^t M \| y(s) \|_Y^2 + M |u(s)|^2 \dd s \Big) e^{M \int_0^t |u(s)|^2 \dd t}.
\end{equation*}
We already have a bound on $\| y_0 \|_Y$. Therefore, by \eqref{eq:bound_uy}, $\| u \|_{L^2(0,T)}$ is bounded and thus the exponential term in the above inequality is bounded. Using again \eqref{eq:bound_uy}, we obtain that there exists a constant $M_5$ (independent of $T$ and $\gamma$) such that
\begin{equation} \label{eq:estima_final_state}
\| y(t) \|_Y \leq M_5 \big( \| y_0 \|_Y + \gamma^2 \| y_0 \|_Y^2 \big)
\leq M_5 \big(1 + \gamma^2 \delta_4 \big) \| y_0 \|_Y, \quad \forall t \in [0,T].
\end{equation}
Now, we fix $\gamma = \max( 2M_5, M_1)$ and reduce the value of $\delta_4$, if necessary, so that $\gamma^2 \delta_4 \leq \frac{1}{2}$.
It follows from \eqref{eq:estima_final_state} that
\begin{equation*}
\| y(T) \|_Y \leq \frac{3}{2} M_5 \| y_0 \|_Y < 2 M_5 \| y_0 \|_Y \leq \gamma \| y_0 \|_Y.
\end{equation*}
This proves that the final-state constraint is not active, therefore, $(y,u)$ is also a local solution to \eqref{eq:Pb_One_step_RH}. Moreover, \eqref{eq:bound_uy}, \eqref{eq:estima_final_state} and the inequality $\gamma^2 \delta_4 \leq \frac{1}{2}$ together yield
\begin{equation*}
\| u \|_{L^2(0,T)} \leq \frac{3}{2} M_3 \| y_0 \|_Y \quad \text{and} \quad
\| y \|_{W(0,T)} \leq \frac{3}{2} \max( M_4, M_5) \| y_0 \|_Y,
\end{equation*}
which concludes the proof, for $k \geq 3$.

The proof is quite similar for $k=2$, therefore we only give the main lines. The main difference is that it is not  necessary anymore to localize the problem with an a-priori final-state constraint, since $\mathcal{V}_2 \geq 0$. As before, one can show that there exists a constant $M > 0$ such that for $\| y_0 \|_Y$ sufficiently small,
$J_{T,k}(\bar{y}_{|(0,T)},\bar{u}_{|(0,T)}) \leq M \| y_0 \|_Y^2$.
Therefore, there exists a minimizing sequence $(y_n,u_n)$ (now directly for Problem \ref{eq:Pb_One_step_RH}) such that
\begin{equation*}
\frac{1}{2} \| Cy_n \|_{L^2(0,T;Z)}^2 + \frac{\alpha}{2} \| u_n \|_{L^2(0,T)}^2
\leq J_{T,k}(y_n,u_n)
\leq M \| y_0 \|_Y^2.
\end{equation*}
Applying Lemma \ref{lemma:observer}, we deduce that $(y_n,u_n)$ is bounded in $W(0,T) \times L^2(0,T)$. We show then that any weak limit point (there exists at least one) is a global solution to the problem and satisfies estimate \eqref{eq:bound_sol_one_step}.
\end{proof}

\begin{lemma} \label{lemma:oc_for_one_step}
Let $\delta_4$ and $M>0$ be given by Proposition \ref{prop:existence_one_step_pb}. There exists $\delta_5 \in (0,\delta_4]$ and $M'>0$ such that for all $y_0 \in B_Y(\delta_5)$ and for all local solutions $(y,u)$ to Problem \eqref{eq:Pb_One_step_RH} satisfying the bound \eqref{eq:bound_sol_one_step}, there exists a unique costate $p \in W(0,T)$, satisfying
\begin{equation} \label{eq:oc_for_one_step_pb}
\begin{cases} \begin{array}{rl}
-\dot{p} - (A + uN)^* p - C^*Cy = \ & \! \! 0, \\
p(T)-D\mathcal{V}_k(y(T)) = \ & \! \! 0, \\
\alpha u + \langle p, Ny + B \rangle_Y = \ & \! \! 0,
\end{array}
\end{cases}
\end{equation}
and  the following bound: $\| p \|_{W(0,T)} \leq M' \| y_0 \|_Y$.
\end{lemma}

\begin{proof}
The costate $p$ is uniquely defined by the first two lines of \eqref{eq:oc_for_one_step_pb}. The well-posedness of this adjoint equation can be studied with the same methods as those used for Lemma \ref{lemma:RegEstim} (see the details of the proof in \cite[Lemma 1]{BreKP17b}). A classical calculation, based on an integration by parts, allows to show the third relation. It follows that the triplet $(y,u,p)$ is the solution to the linear system \eqref{eq:non_reg_os}, where
\begin{equation*}
(f,g,h,q)= (0,uN^*p, \langle Ny,p \rangle_Y,D\mathcal{V}_k(y(T)) - \Pi y(T))
\end{equation*}
and $Q= \Pi$.
We have
\begin{align*}
\| g \|_{L^2(0,T;V^*)} \leq \ & \| u \|_{L^2(0,T)} \| N \|_{\mathcal{L}(Y,V^*)} \| p \|_{L^\infty(0,T;Y)}
\leq M \| y_0 \|_Y \| p \|_{W(0,T)}, \\
\| h \|_{L^2(0,T)} \leq \ & \| N \|_{\mathcal{L}(V,Y)} \| y \|_{L^2(0,T;V)} \| p \|_{L^\infty(0,T;Y)}
\leq M \| y_0 \|_Y \| p \|_{W(0,T)}, \\
\| q \|_Y \leq \ & M \| y(T) \|_Y, \text{ and } \| q \|_Y \leq M \| y(0) \|_Y.
\end{align*}
Therefore, there exists a constant $M_1 > 0$, independent of $T$, such that
\begin{equation*}
\| (y_0,f,g,h,q) \|_{\Upsilon_{T,0}} \leq M_1 \| y_0 \|_Y (1 + \| p \|_{W(0,T)}).
\end{equation*}
Let us denote by $M_2$ the constant involved in \eqref{eq:non_reg_os}. We obtain with Proposition \ref{proposition:non_reg_os} that
\begin{equation*}
\| p \|_{W(0,T)} \leq M_1 M_2 \| y_0 \|_Y + M_1 M_2 \| y_0 \|_Y \| p \|_{W(0,T)}.
\end{equation*}
The announced bound on $p$ follows, taking $\delta_5= \min (\delta_4, (2M_1M_2)^{-1})$ and $M'= 2M_1M_2$.
\end{proof}

We are now ready to prove an estimate for $\| (y_{T,k},u_{T,k},p_{T,k}) - (\bar{y},\bar{u},\bar{p}) \|$, by ``comparing" the associated optimality conditions and applying the inverse function theorem.

\begin{proposition} \label{prop:estimate_one_step}
There exist $\delta_6 \in (0,\delta_5]$, $\delta_6' > 0$, and $M > 0$ such that for all $y_0 \in B_Y(\delta_6)$, Problem \eqref{eq:Pb_One_step_RH} has a unique local solution $(y_{T,k},u_{T,k})$ with associated costate $p_{T,k}$ satisfying
\begin{equation*}
\| (y_{T,k},u_{T,k},p_{T,k}) \|_{\Lambda_{T,0}}
\leq \delta_6'.
\end{equation*}
Moreover,
\begin{equation} \label{eq:estimate_rhc}
\begin{cases}
\begin{array}{rl}
\| (y_{T,k},u_{T,k},p_{T,k}) - (\bar{y},\bar{u},\bar{p}) \|_{\Lambda_{T,0}}
\leq & \! \! \! M \| \bar{y}(T) \|_Y^k, \\
\| (y_{T,k},u_{T,k},p_{T,k}) - (\bar{y},\bar{u},\bar{p}) \|_{\Lambda_{T,-\lambda}} \leq & \! \! \! M e^{-\lambda T} \| \bar{y}(T) \|_Y^k,
\end{array}
\end{cases}
\end{equation}
where $\bar{y}$, $\bar{u}$, and $\bar{p}$ are the restrictions of $\mathcal{Y}_1(y_0)$, $\mathcal{U}_1(y_0)$, and $\mathcal{P}_1(y_0)$ to $(0,T)$.
\end{proposition}

\begin{proof}
\emph{Step 1:} construction of $\Phi_2$ and application of the inverse mapping theorem.\\
Consider the mapping $\Phi_2$, defined as follows:
\begin{equation} \label{eq:defPhi2}
\Phi_2 \colon
(y,u,p) \in \Lambda_{T,0} \mapsto
\begin{pmatrix}
y(0) \\ \dot{y}- (Ay + (Ny+B)u) \\
- \dot{p} - A^*p - u N^* p - C^*C y \\
\alpha u + \langle Ny+ B,p \rangle_Y \\
p(T)- D\mathcal{V}_k(y(T))
\end{pmatrix}
\in \Upsilon_{T,0}.
\end{equation}
The reader can check that the mapping $\Phi_2$, considered from $\Lambda_{T,0}$ to $\Upsilon_{T,0}$ is differentiable, with a Lipschitz-continuous derivative, in a neighborhood of $(0,0,0)$. The size of the neighborhood and the Lipschitz-modulus can be both chosen independently of $T$. One can also prove that the mapping $\Phi_2$, considered from $\Lambda_{T,-\lambda}$ to $\Upsilon_{T,-\lambda}$ is differentiable and that there exist $\delta > 0$ and $M >0$ such that for all $(y,u,p)$ and $(\tilde{y},\tilde{u},\tilde{p}) \in B_{\Lambda_{T,0}}(\delta)$,
\begin{equation} \label{eq_strange_lip_cont}
\| D \Phi_2(\tilde{y},\tilde{u},\tilde{p}) - D \Phi_2(y,u,p) \|_{\mathcal{L}(\Lambda_{T,-\lambda};\Upsilon_{T,-\lambda})}
\leq M \| (\tilde{y},\tilde{u},\tilde{p}) - (y,u,p) \|_{\Lambda_{T,0}}.
\end{equation}
Some elements of proof concerning the regularity of $\Phi_2$ are given in the Appendix. See also Remarks \ref{remark:19} and \ref{remark:why_new_ift} on the necessity to apply the extension of the implicit function theorem given in Theorem \ref{thmInvThm}.
By Proposition \ref{proposition:non_reg_os}, the derivative $D\Phi_2(0,0,0)$, seen as an element of $\mathcal{L}(\Lambda_{T,0};\Upsilon_{T,0})$ and of $\mathcal{L}(\Lambda_{T,-\lambda};\Upsilon_{T,-\lambda})$,
has a bounded inverse. Moreover there exists $M>0$ such that
\begin{equation*}
\| D\Phi_2(0,0,0)^{-1} \|_{\mathcal{L}(\Upsilon_{T,0};\Lambda_{T,0})} \leq M \quad \text{and} \quad
\| D\Phi_2(0,0,0)^{-1} \|_{\mathcal{L}(\Upsilon_{T,-\lambda};\Lambda_{T,-\lambda})} \leq M.
\end{equation*}
Therefore, by the inverse mapping theorem, there exist $\delta_6 > 0$, $\delta_6' >0$, $M>0$ (all independent of $T$), and three mappings
\begin{equation*}
(y_0,q) \in B_Y(\delta_6)^2 \mapsto (\mathcal{Y}_2,\mathcal{U}_2,\mathcal{P}_2)(y_0,q)
\in \Lambda_{T,0}
\end{equation*}
such that for all $(y_0,q) \in B_Y(\delta_6)^2$, the triplet $(\mathcal{Y}_2,\mathcal{U}_2,\mathcal{P}_2)(y_0,q)$ is the unique solution to
\begin{equation} \label{eq:nonlin_eq_phi2}
\Phi_2(y,u,p) = (y_0,0,0,0,q), \quad
\| (y,u,p) \|_{\Lambda_{T,0}} \leq \delta_6'.
\end{equation}
The mappings $\mathcal{Y}_2$, $\mathcal{U}_2$, and $\mathcal{P}_2$ are Lipschitz-continuous in the following sense: for all $(y_0,q) \in B_Y(\delta_6)^2$ and $(\tilde{y}_0,\tilde{p}) \in B_Y(\delta_6)^2$,
\begin{equation} \label{eq_lipschA}
\begin{cases}
\begin{array}{rl}
\| (\mathcal{Y}_2,\mathcal{U}_2,\mathcal{P}_2)(\tilde{y}_0,\tilde{q})
- (\mathcal{Y}_2,\mathcal{U}_2,\mathcal{P}_2)(y_0,q) \|_{\Lambda_{T,0}}
\leq & \! \! \! \max \big( \| \tilde{y}_0 - y_0 \|_Y, \| \tilde{q} - q \|_Y \big) \\
\| (\mathcal{Y}_2,\mathcal{U}_2,\mathcal{P}_2)(\tilde{y}_0,\tilde{q})
- (\mathcal{Y}_2,\mathcal{U}_2,\mathcal{P}_2)(y_0,q) \|_{\Lambda_{T,-\lambda}}
\leq & \! \! \! \max \big( \| \tilde{y}_0 - y_0 \|_Y, e^{-\lambda T} \| \tilde{q} - q \|_Y \big).
\end{array}
\end{cases}
\end{equation}
The weight $e^{-\lambda T}$ comes here from the weight used in front of the variable $q$ in the definition of $\Upsilon_{T,-\lambda}$.
All along the proof, the value of $\delta_6$ is reduced. Let us emphasize the fact that the new values of $\delta_6$ can all be chosen independently of $T$.

\emph{Step 2:} characterization of $(y_{T,k},u_{T,k},p_{T,k})$.\\
By $M_1$ and $M_2$  we denote the constants involved in Proposition \ref{prop:existence_one_step_pb} and Lemma \ref{lemma:oc_for_one_step}, respectively. Let us reduce the value of $\delta_6$, if necessary, so that $\delta_6 \leq \min(\delta_5, \delta_6'/M_1, \delta_6'/M_2)$. For all $y_0 \in B_Y(\delta_6)$, there exists a solution $(y_{T,k},u_{T,k})$ to Problem \eqref{eq:Pb_One_step_RH} with associated costate $p_{T,k}$ such that
\begin{equation*}
\| (y_{T,k},u_{T,k},p_{T,k}) \|_{\Lambda_{T,0}} \leq \max(M_1,M_2) \delta_6 \leq \delta_6'.
\end{equation*}
Moreover $\Phi_2(y_{T,k},u_{T,k},p_{T,k})= (y_0,0,0,0,0)$  by Lemma \ref{lemma:oc_for_one_step}. This proves that $(y_{T,k},u_{T,k},p_{T,k})$ is the unique solution to \eqref{eq:nonlin_eq_phi2} and therefore that
\begin{equation} \label{eq_charac1}
(y_{T,k},u_{T,k},p_{T,k})= (\mathcal{Y}_2,\mathcal{U}_2,\mathcal{P}_2)(y_0,0).
\end{equation}
This also proves the (local) uniqueness of local solutions to \eqref{eq:Pb_One_step_RH}.

\emph{Step 3:} characterization of $(\bar{y},\bar{u},\bar{p})$.\\
The polynomial function $\mathcal{V}_k$ is a Taylor approximation of order $k$ of the value function. Therefore, $D\mathcal{V}_k$ is a Taylor approximation of order $k-1$ of $D\mathcal{V}$. As a consequence, there exist $M_3$ and $\delta > 0$ such that for all $y \in B_Y(\delta)$,
\begin{equation} \label{eq:taylor_k}
\| D \mathcal{V}_k(y)- D \mathcal{V}(y) \|_Y \leq M_3 \| y \|_Y^k.
\end{equation}
If necessary, we reduce $\delta$ so that $M_3 \delta^k \leq \delta_6'$.
We reduce then the value of $\delta_6$, if necessary, so that $\delta_6 \leq \delta_2$ and so that  $\| \mathcal{Y}_1(y_0) \|_{L^\infty(0,\infty;Y)} \leq \delta$ for all $y_0 \in B_Y(\delta_6)$.
Let $y_0 \in B_Y(\delta_6)$. Let us denote by $\bar{y}$, $\bar{u}$ and $\bar{p}$ the restrictions to $(0,T)$ of $\mathcal{Y}_1(y_0)$, $\mathcal{U}_1(y_0)$, and $\mathcal{P}_1(y_0)$.
As a consequence of Lemma \ref{lemma:sensitivity_relation}, we have
\begin{equation*}
\Phi_2(\bar{y},\bar{u},\bar{p})
= (y_0,0,0,0,q),
\end{equation*}
with $q= D\mathcal{V}_k(\bar{y}(T))- D\mathcal{V}(\bar{y}(T))$. By \eqref{eq:taylor_k}, we have $\| q \|_Y \leq M_3 \| \bar{y}(T) \|^k_Y \leq M_3 \delta^k \leq \delta_6'$.
Since the mappings $\mathcal{Y}_1$, $\mathcal{U}_1$, and $\mathcal{P}_1$ are Lipschitz continuous, the value of $\delta_6$ can be reduced, for the last time, so that $\| ( \bar{y},\bar{u},\bar{p}) \|_{\Lambda_{T,0}} \leq \delta_6'$.
Therefore, $(\bar{y},\bar{u},\bar{p})$ is the unique solution to \eqref{eq:nonlin_eq_phi3} and thus
\begin{equation} \label{eq_charac2}
(\bar{y},\bar{u},\bar{p})= (\mathcal{Y}_2,\mathcal{U}_2,\mathcal{P}_2)(y_0,q).
\end{equation}

\emph{Step 4:} proof of estimate \eqref{eq:estimate_rhc}.\\
Combining \eqref{eq_lipschA}, the definition of $q$, and \eqref{eq:taylor_k}, we obtain that
\begin{align*}
\| (\mathcal{Y}_2,\mathcal{U}_2,\mathcal{P}_2)(y_0,0)
- (\mathcal{Y}_2,\mathcal{U}_2,\mathcal{P}_2)(y_0,q) \|_{\Lambda_{T,0}}
\leq \ & M \| \bar{y}(T) \|_Y^k \\
\| (\mathcal{Y}_2,\mathcal{U}_2,\mathcal{P}_2)(\tilde{y}_0,\tilde{q})
- (\mathcal{Y}_2,\mathcal{U}_2,\mathcal{P}_2)(y_0,q) \|_{\Lambda_{T,-\lambda}} \leq \ &  M e^{-\lambda T} \| \bar{y}(T) \|_Y^k.
\end{align*}
Estimate \eqref{eq:estimate_rhc} follows, using the characterizations \eqref{eq_charac1} and \eqref{eq_charac2}.
\end{proof}

In the sequel, the triplet $(\mathcal{Y}_2,\mathcal{U}_2,\mathcal{P}_2)(y_0)$ indicates the solution (with its associated costate) to \eqref{eq:Pb_One_step_RH}. The triplet is also denoted $(y_{T,k},u_{T,k},p_{T,k})$ when no ambiguity is possible.

\begin{proposition} \label{prop:other_estim_one_step}
There exist $\delta_7 \in (0,\delta_6]$ and $M > 0$ such that for all $y_0$ and $\tilde{y}_0 \in B_Y(\delta_6)$,
\begin{equation} \label{eq:estimate1}
\| \mathcal{Y}_1(\tilde{y}_0)-\mathcal{Y}_1(y_0) \|_{W_\lambda(0,T)} \leq M \| \tilde{y}_0- y_0 \|_Y.
\end{equation}
Moreover, for all $T \geq T_0$,
\begin{equation} \label{eq:estimate2}
\| \mathcal{Y}_2(\tilde{y}_0,0)-\mathcal{Y}_2(y_0,0) \|_{W_\lambda(0,T)} \leq M \| \tilde{y}_0- y_0 \|_Y.
\end{equation}
\end{proposition}

\begin{remark} \label{remark:speed_cv}
As a direct consequence of the above proposition, we obtain that for all $y_0 \in B_Y(\delta_7)$, for all $t \in [0,\infty)$,
\begin{equation}
\| \bar{y}(t) \|_Y \leq M e^{-\lambda t} \| y_0 \|_Y,
\end{equation}
where $\bar{y}$ is the optimal trajectory. Moreover, for all $t \in [0,T]$,
$\| y_{T,k}(t) \|_Y \leq M e^{-\lambda t} \| y_0 \|_Y,
$
where $y_{T,k}$ denotes the optimal trajectory associated with the solution to \eqref{eq:Pb_One_step_RH}.
\end{remark}

\begin{proof}[Proof of Proposition \ref{prop:other_estim_one_step}]
\emph{Step 1:} construction of the mapping $\Phi_3$ and application of the inverse mapping theorem.\\
Consider the mapping $\Phi_3$, defined as $\Phi_2$ but from $\Lambda_{T,\lambda}$ to $\Upsilon_{T,\lambda}$.
We let the reader check that $\Phi_3$ is well-defined, differentiable, with a locally Lipschitz-continuous derivative.
By Proposition \ref{proposition:non_reg_os}, $D\Phi_3(0,0,0)$ has a bounded inverse. Moreover, there exists $M>0$ such that $\| D \Phi_3(0,0,0)^{-1} \| \leq M$, for all $T \geq T_0$. Therefore, by the inverse mapping theorem, there exist $\delta_7>0$, $\delta_7'>0$, $M>0$ (independent of $T$), and three $M$-Lipschitz continuous mappings
\begin{equation*}
y_0 \in B_Y(\delta_7) \mapsto (\mathcal{Y}_3,\mathcal{U}_3,\mathcal{P}_3)(y_0)
\in \Lambda_{T,\lambda}
\end{equation*}
such that for all $y_0 \in B_Y(\delta_7)$, the triplet $(\mathcal{Y}_3,\mathcal{U}_3,\mathcal{P}_3)(y_0)$ is the unique solution to
\begin{equation} \label{eq:nonlin_eq_phi3}
\Phi_3(y,u,p) = (y_0,0,0,0,0), \quad
\| (y,u,p) \|_{\Lambda_{T,\lambda}} \leq \delta_7'.
\end{equation}
As in the proof of the previous proposition, the value of $\delta_7$ will be reduced, still the new values of $\delta_7$ can be chosen independently of $T$.

\emph{Step 2:} the mappings $\Phi_2$ and $\Phi_3$ coincide.\\
Let us reduce $\delta_7>0$, if necessary, so that $\delta_7 \leq \delta_6$. By Lemma \ref{lemma:embedding}, for all $y_0 \in B_Y(\delta_7)$,
\begin{equation*}
\| (\mathcal{Y}_3,\mathcal{U}_3,\mathcal{P}_3)(y_0) \|_{\Lambda_{T,0}}
\leq M \| (\mathcal{Y}_3,\mathcal{U}_3,\mathcal{P}_3)(y_0) \|_{\Lambda_{T,\lambda}},
\end{equation*}
where the constant $M$ is independent of $T$. Reducing $\delta_7$ so that $\delta_7 \leq \delta_6'/M$, we obtain that for all $y_0 \in B_Y(\delta_7)$, $\| (\mathcal{Y}_3,\mathcal{U}_3,\mathcal{P}_3)(y_0,q) \|_{\Lambda_{T,0}}
\leq \delta_6'$.
We also have
$\Phi_2((\mathcal{Y}_3,\mathcal{U}_3,\mathcal{P}_3)(y_0))
= (y_0,0,0,0,0)$.
Since $\| y_0 \|_Y \leq \delta_6$, we obtain that $(\mathcal{Y}_3,\mathcal{U}_3,\mathcal{P}_3)(y_0)$ is the unique solution to \eqref{eq:nonlin_eq_phi2} (with $q=0$) and finally that
$(\mathcal{Y}_3,\mathcal{U}_3,\mathcal{P}_3)(y_0)
= (\mathcal{Y}_2,\mathcal{U}_2,\mathcal{P}_2)(y_0,0)$.
The estimate \eqref{eq:estimate2} follows, using the Lipschitz-continuity of $\mathcal{Y}_3$ for the $W_\lambda(0,T)$-norm.

\emph{Step 3:} proof of estimate \eqref{eq:estimate1}.\\
Estimate \eqref{eq:estimate1} can be proved in a very similar way to \eqref{eq:estimate2}, therefore, we only sketch the proof.
Consider the mapping $\Phi_4$, defined as follows:
\begin{equation*}
\Phi_4 \colon
(y,u,p) \in \Lambda_{T,\lambda} \mapsto
\begin{pmatrix}
y(0) \\ \dot{y}- (Ay + (Ny+B)u) \\
- \dot{p} - A^*p - u N^* p - C^*C y \\
\alpha u + \langle Ny+ B,p \rangle_Y \\
p(T)- D\mathcal{V}(y(T))
\end{pmatrix}
\in \Upsilon_{T,\lambda}.
\end{equation*}
Applying the inverse function theorem (using in particular Proposition \ref{proposition:non_reg_os}), one obtains three Lipschitz-continuous mapping $\mathcal{Y}_4$, $\mathcal{U}_4$, and $\mathcal{P}_4$. Then, one can show that these mappings locally coincide with $\mathcal{Y}_1$, $\mathcal{U}_1$, and $\mathcal{P}_1$, respectively.
Estimate \eqref{eq:estimate1} follows.
\end{proof}

The following corollary collects the different estimates that will be used in the analysis of the last section.

\begin{corollary} \label{corollary}
There exists a constant $M>0$ such that for all $y_0$ and $\tilde{y}_0 \in B_Y(\delta_7)$, for all $\tau$ and $T$ with $0 \leq \tau \leq T$,
\begin{align*}
\max \big( \| \mathcal{Y}_1(\tilde{y}_0) - \mathcal{Y}_1(y_0) \|_{W(0,\tau)} , \| \mathcal{U}_1(\tilde{y}_0)-\mathcal{U}_1(y_0) \|_{L^2(0,\tau)} \big) \leq \ & M \| \tilde{y}_0-y_0 \|_Y \label{a} \tag{$a$} \\
\| e^{\lambda \tau} ( \mathcal{Y}_1(\tilde{y}_0)-\mathcal{Y}_1(y_0)) \|_{Y} \leq \ & M \| \tilde{y}_0- y_0 \|_Y \tag{$b$} \label{b} \\
\| y_{T,k}(\tau) \|_Y \leq \ & M e^{-\lambda \tau} \| y_0 \|_Y \tag{$c$} \label{c} \\
\max( \| y_{T,k} - \bar{y} \|_{W(0,\tau)} , \| u_{T,k} - \bar{u} \|_{L^2(0,\tau)} ) \leq \ & M e^{\lambda \tau -\lambda (k+1) T } \| y_0 \|_Y^k \tag{$d$} \label{d} \\
\| y_{T,k}(\tau)-\bar{y}(\tau) \|_{Y} \leq \ & M e^{\lambda \tau - (k+1) \lambda T} \| y_0 \|_Y^k. \tag{$e$} \label{e}
\end{align*}
\end{corollary}

\begin{proof}
Estimate \eqref{a} is proved in Proposition \ref{prop:existenceSol}. Estimates \eqref{b} and \eqref{c} are proved in Proposition \ref{prop:other_estim_one_step}.
By Proposition \ref{prop:other_estim_one_step}, we have $\| \bar{y}(t) \|_Y \leq M e^{-\lambda t} \| y_0 \|_Y$.
Combined with Proposition \ref{prop:estimate_one_step}, we obtain that
\begin{equation*}
\| (y_{T,k},u_{T,k},p_{T,k}) -(\bar{y},\bar{u},\bar{p}) \|_{\Lambda_{T,-\lambda}} \leq M e^{-\lambda (k+1) T } \| y_0 \|_Y^k.
\end{equation*}
We obtain then with Lemma \ref{lemma:embedding} that
\begin{align*}
\| (y_{T,k},u_{T,k},p_{T,k}) -(\bar{y},\bar{u},\bar{p}) \|_{\Lambda_{\tau,0}}
\leq \ & e^{\lambda \tau} \| (y_{T,k},u_{T,k},p_{T,k}) -(\bar{y},\bar{u},\bar{p}) \|_{\Lambda_{\tau,-\lambda}} \\
\leq \ & M e^{\lambda \tau -\lambda (k+1) T } \| y_0 \|_Y^k.
\end{align*}
from which estimates \eqref{d} and \eqref{e} immediately follows.
\end{proof}

\section{Error estimates for the Receding-Horizon method} \label{section:rhc_estimate}

\begin{proof}[Proof of Theorem \ref{thm:main}]

We fix now $\tau_0 \geq 0$ such that
\begin{equation*}
r_0:= Me^{-\lambda \tau_0} < 1,
\end{equation*}
where $M$ is the constant provided in Corollary \ref{corollary}.
We make use of the following notation:
\begin{equation*}
r= Me^{-\lambda \tau}, \quad
\theta= e^{\lambda \tau -\lambda (k+1) T} \| y_0 \|_Y^k, \quad
t_n= n \tau, \quad
\bar{y}_n= \bar{y}(n\tau),
\end{equation*}
where $T >\tau \ge \tau_0$. Note that $r \leq r_0 < 1$.

\emph{Step 1:} well-posedness of the algorithm.\\
Let us prove by induction that for all $n \in \mathbb{N}$, the algorithm is well-posed at steps 0, 1,...,$n-1$ and that $y_n \in B_Y(\delta_7)$. For $n= 0$, the statement is true by assumption. Assume that it holds for a given $n$. Since $y_n \in B_Y(\delta_7)$,  by Proposition \ref{prop:estimate_one_step}, Problem \eqref{eq:Pb_One_step_RH} with initial condition $y_n$ has a unique local solution $y_{T,k}$. Moreover, by estimate \eqref{c},
$\| y_{n+1} \|_Y = \| y_{RH}((n+1)\tau) \|_Y
= \| y_{T,k}(\tau) \|_Y
\leq r \| y_n \|_Y \leq \delta_7$.
Therefore, the statement holds at $(n+1)$, which concludes the proof of well-posedness. Note that a direct consequence of the last inequality is that
\begin{equation} \label{eq:decay_y_n}
\| y_n \|_Y \leq r^n \| y_0 \|_Y, \quad \text{for all $n \in \mathbb{N}$}.
\end{equation}

\emph{Step 2}: estimation of $\| y_{RH}-\bar{y} \|_{W(0,\infty)}$ and $\| u_{RH}-\bar{u} \|_{L^2(0,\infty)}$.\\
Consider the following sequences:
\begin{align*}
a_n = \ & \max \big( \| y_{RH} - \bar{y} \|_{W(t_n,t_{n+1})}, \| u_{RH}- \bar{u} \|_{L^2(t_n,t_{n+1})} \big) \\
b_n = \ & \| y_n - \bar{y}_n \|_Y.
\end{align*}
We prove in this second step that for all $n \in \mathbb{N}$,
\begin{align}
a_n \leq \ & M \theta r^{kn} + M b_n, \label{eq:estiM_1} \\
b_{n+1} \leq \ & M \theta r^{kn} + r b_n. \label{eq:estima_b}
\end{align}
Before proving these two estimates, observe that by Proposition \ref{prop:other_estim_one_step}, for all $n \in \mathbb{N} \backslash \{ 0 \}$,
\begin{equation*}
\| \bar{y}_n \|_Y \leq M e^{-\lambda t_n} \| y_0 \|_Y \leq M e^{-\lambda \tau} \| y_0 \|_Y \leq \| y_0 \|_Y \leq \delta_7.
\end{equation*}
Of course, we also have $\| \bar{y}_0 \|_Y = \| y_0 \|_Y \leq \delta_7$.

Let $n \in \mathbb{N}$ and let us prove \eqref{eq:estiM_1}.
Let $(y_{T,k},u_{T,k})$ be the local solution to \eqref{eq:Pb_One_step_RH} with initial condition $y_n$ (characterized in Proposition \ref{prop:estimate_one_step}).
Recall that by construction, $y_{RH}(t_n + t)= y_{T,k}(t)$, for $t \in (0,\tau)$.
Moreover, by dynamic programming, $\bar{y}(t_n+t)= \mathcal{Y}_1(\bar{y}_n;t)$, for \@ $t \in (0,\tau)$.
Therefore,
\begin{align*}
\| y_{RH}- \bar{y} \|_{W(t_n,t_{n+1})} = \ & \| y_{T,k} - \mathcal{Y}_1(\bar{y}_n) \|_{W(0,\tau)} \\
\leq \ & \| y_{T,k} - \mathcal{Y}_1(y_n) \|_{W(0,\tau)} + \| \mathcal{Y}_1(y_n)- \mathcal{Y}_1(\bar{y}_n) \|_{W(0,\tau)}.
\end{align*}
Using estimate \eqref{d} and \eqref{eq:decay_y_n}, we obtain
$\| y_{T,k} - \mathcal{Y}_1(y_n) \|_{W(0,\tau)}
\leq Me^{\lambda \tau -\lambda (k+1)T} \| y_n \|_Y^k
\leq M \theta r^{kn}$.
Using estimate \eqref{a}, we find
$\| \mathcal{Y}_1(y_n)- \mathcal{Y}_1(\bar{y}_n) \|_{W(0,\tau)}
\leq M \| y_n- \bar{y}_n \|_Y = M b_n$.
Combining the last three obtained estimates, we obtain that
\begin{equation*}
\| y_{RH}- \bar{y} \|_{W(t_n,t_{n+1})} \leq M \theta r^{kn} + M b_n.
\end{equation*}
The term $\| u_{RH} - \bar{u} \|_{L^2(t_n,t_{n+1})}$ can be estimated exactly in the same way. Estimate \eqref{eq:estiM_1} follows.

Estimate \eqref{eq:estima_b} can be proved similarly.
We have
\begin{equation*}
\| y_{n+1}- \bar{y}_{n+1} \|_Y \le \| y_{T,k}(\tau) - \mathcal{Y}_1(y_n;\tau) \|_Y + \| \mathcal{Y}_1(y_n;\tau)- \mathcal{Y}_1(\bar{y}_n;\tau) \|_Y.
\end{equation*}
Using estimate \eqref{e} and \eqref{eq:decay_y_n}, we obtain
\begin{equation*}
\| y_{T,k}(\tau) - \mathcal{Y}_1(y_n;\tau) \|_Y
\leq M e^{\lambda \tau - \lambda(k+1)T} \| y_n \|_Y^k
\leq M \theta r^{kn}.
\end{equation*}
Using estimate $(b)$, we obtain that
\begin{equation*}
\| \mathcal{Y}_1(y_n;\tau)- \mathcal{Y}_1(\bar{y}_n;\tau) \|_Y
\leq M e^{-\lambda \tau} \| y_n- \bar{y}_n \|_Y = r b_n.
\end{equation*}
Combining the last three obtained estimates, we obtain \eqref{eq:estima_b}.

\emph{Step 3:} proof of estimate \eqref{eq:final_estim1}.\\
Let us set $c_n= b_n/r^{n-1}$. By \eqref{eq:estima_b}, we have
\begin{equation*}
c_{n+1} \leq M \theta r^{(k-1)n} + c_n \leq M \theta + c_n,
\end{equation*}
since $k \geq 2$. We have $c_0=b_0= 0$, therefore $c_n \leq n M \theta$ and
$b_n \leq M \theta n r^{n-1}$.
Moreover, $a_n \leq M \theta \big( r^{kn} + nr^{n-1} \big)$
and finally
\begin{align*}
& \max \big( \| y_{RH} - \bar{y} \|_{W(0,\infty)}, \| u_{RH} - \bar{u} \|_{L^2(0,\infty)} \big) \leq \sum_{n=0}^\infty a_n \\
& \qquad \leq M \theta \sum_{n=0}^\infty \big( r^{kn} + nr^{n-1} \big)
= M \theta \Big( \frac{1}{1-r^k} + \frac{1}{(1-r)^2} \Big)
\leq M \theta \Big( \frac{1}{1-r_0^k} + \frac{1}{(1-r_0)^2} \Big),
\end{align*}
which proves \eqref{eq:final_estim1}.

\emph{Step 4: proof of estimate \eqref{eq:final_estim2}}. \\
In the following equalities, we denote the norms $\| \cdot \|_{L^2(0,\infty;Z)}$ and $\| \cdot \|_{L^2(0,\infty)}$ by $\| \cdot \|$ to simplify. We have
\begin{align}
\mathcal{J}(u_{RH},y_0)- \mathcal{V}({y}_0)
= \ & \Big( \frac{1}{2} \| C y_{RH} \|^2 + \frac{\alpha}{2} \| u_{RH} \|^2 \Big) -
\Big( \frac{1}{2} \| C \bar{y} \|^2 + \frac{\alpha}{2} \| \bar{u} \|^2 \Big) \notag \\
& \qquad - \big\langle \bar{p}, \dot{y}_{RH}-(Ay_{RH}+(Ny_{RH}+B)u_{RH}) \big\rangle_{L^2(0,\infty;V),L^2(0,\infty;V^*)} \notag \\
& \qquad + \big\langle \bar{p}, \dot{\bar{y}}- (A\bar{y} + (N\bar{y} + B)\bar{u}) \big\rangle_{L^2(0,\infty;V),L^2(0,\infty;V^*)}. \label{eq:diffValFunc}
\end{align}
Indeed, the last two terms (in brackets) are null.
The four following relations can be easily verified:
\begin{equation} \label{eq:diffValFunc2}
\begin{aligned}
\frac{1}{2} \| Cy_{RH} \|^2 - \frac{1}{2} \| C \bar{y} \|^2
= \ & \langle C^*C \bar{y}, y_{RH}- \bar{y} \rangle_{L^2(0,\infty;Y)} + \frac{1}{2} \| C(y_{RH}-\bar{y}) \|^2, \\
\frac{\alpha}{2} \| u_{RH} \|^2 - \frac{\alpha}{2} \| \bar{u} \|^2
= \ & \alpha \langle \bar{u}, u_{RH}- \bar{u} \rangle_{L^2(0,\infty)} + \frac{\alpha}{2} \| u_{RH} - \bar{u} \|^2, \\
Ny_{RH} u_{RH} - N \bar{y} \bar{u}
= \ & N \bar{y}(u_{RH}- \bar{u}) + N(y_{RH}-\bar{y})\bar{u} + N(y_{RH}-\bar{y})(u_{RH}-\bar{u}), \\
-\langle \bar{p},\dot{y}_{RH} -\dot{\bar{y}} \rangle_{L^2(0,\infty;V),L^2(0,\infty;V^*)}
= \ & \langle \dot{\bar{p}}, y_{RH}- \bar{y} \rangle_{L^2(0,\infty;V^*),L^2(0,\infty;V)}.
\end{aligned}
\end{equation}
Combining \eqref{eq:diffValFunc} and \eqref{eq:diffValFunc2} yields
\begin{align*}
\mathcal{J}(u_{RH},y_0)-\mathcal{V}(y_0)
= \ &  \frac{1}{2} \| C(y_{RH}-\bar{y}) \|^2
+ \frac{\alpha}{2} \| u_{RH}- \bar{u} \|^2 \\
& \quad + \big\langle \bar{p}, N(y_{RH}-\bar{y})(u_{RH}-\bar{u}) \big\rangle_{L^2(0,\infty;V);L^2(0,\infty;V^*)} \\
& \quad + \big\langle \underbrace{ \dot{\bar{p}} + A^*\bar{p} + \bar{u} N^* \bar{p} + C^* C \bar{y} }_{=0}, y_{RH}- \bar{y} \big\rangle_{L^2(0,\infty;V^*);L^2(0,\infty;V)} \\
& \quad + \big\langle \underbrace{\alpha \bar{u} + \langle N \bar{y}+ B, \bar{p} \rangle_Y}_{=0}, u_{RH}- \bar{u} \big\rangle_{L^2(0,\infty)}.
\end{align*}
The three remaining quadratic terms (on the right-hand side) can be estimated with \eqref{eq:final_estim1}. We finally obtain
\begin{align*}
\mathcal{J}(u_{RH},y_0)-\mathcal{V}(y_0)
\leq \ & M \max( \| y_{RH} - \bar{y} \|_{W(0,\infty)}, \| u_{RH} - \bar{u} \|_{L^2(0,\infty)} )^2 \\
\leq \ & M \big( e^{-\lambda(T-\tau)- \lambda kT} \| y_0 \|_Y^k \big)^2,
\end{align*}
as was to be proved.
\end{proof}

\section{The case of quadratic terminal cost functions} \label{section:quad_tc}

In this section, we extend our analysis to the situation of a terminal penalty cost which is a non-negative quadratic functional. A particular case is the one of a zero penalty, which can be seen as a first-order Taylor expansion of the value function.
Let us fix a bounded set $\mathcal{Q}$ (in $\mathcal{L}(Y)$) of symmetric and positive semi-definite operators.
Problem \eqref{eq:Pb_One_step_RH} is now replaced by the following one in the design of an RHC method:
\begin{equation*} \label{eq:Pb_One_step_RH_quad} \tag{$P_{T,Q}$}
\begin{cases} \begin{array}{l}
{\displaystyle \inf_{(y,u) \in W(0,T) \times L^2(0,T)} \
\frac{1}{2} \int_0^T \| Cy(t) \|_{Z}^2 \dd t + \frac{\alpha}{2} \int_0^T u(t)^2 \dd t +
\frac{1}{2} \langle y(T), Q y(T) \rangle_Y,} \\[1em]
\qquad \text{subject to: } \dot{y}= Ay + Nyu + Bu, \quad y(0)= y_0,
\end{array}
\end{cases}
\end{equation*}
where $Q \in \mathcal{Q}$.
The analysis which has been done in Sections \ref{section:finite-horizon} and \ref{section:rhc_estimate} can be adapted to this new class of terminal cost functions without difficulty. In order to prove Theorem \ref{thm:main_bis} below, we simply comment on the modifications which have to be realized.
First the existence of a global solution to \eqref{eq:Pb_One_step_RH_quad} can be established, assuming that $\| y_0 \|_Y$ is sufficiently small. The proof is the same as the one of Proposition \ref{prop:existence_one_step_pb} (in the case $k=2$). One can then derive optimality conditions. They have the same form as in Lemma \ref{lemma:oc_for_one_step}, but with another terminal condition:
\begin{equation*}
p_{T,Q}(T) = Q y_{T,Q}(T).
\end{equation*}
Proposition \ref{prop:estimate_one_step} has to adapted as follows.

\begin{proposition} \label{prop:estimate_one_step_2}
There exist $\delta >0$, $\delta' > 0$, and $M > 0$ such that for all $y_0 \in B_Y(\delta)$ and for all $Q \in \mathcal{Q}$, Problem \eqref{eq:Pb_One_step_RH} has a unique local solution $(y_{T,Q},u_{T,Q})$ with associated costate $p_{T,Q}$ satisfying
\begin{equation*}
\| (y_{T,Q},u_{T,Q},p_{T,Q}) \|_{\Lambda_{T,0}}
\leq \delta'.
\end{equation*}
Moreover,
\begin{equation*}
\begin{cases}
\begin{array}{rl}
\| (y_{T,Q},u_{T,Q},p_{T,Q}) - (\bar{y},\bar{u},\bar{p}) \|_{\Lambda_{T,0}}
\leq & \! \! \! M  \big( \| Q - \Pi \|_{\mathcal{L}(Y)} + \| y(T) \|_Y \big) \| y(T) \|_Y, \\
\| (y_{T,Q},u_{T,Q},p_{T,Q}) - (\bar{y},\bar{u},\bar{p}) \|_{\Lambda_{T,-\lambda}} \leq & \! \! \! M e^{-\lambda T} \big( \| Q - \Pi \|_{\mathcal{L}(Y)} + \| y(T) \|_Y \big) \| y(T) \|_Y,
\end{array}
\end{cases}
\end{equation*}
where $\bar{y}$, $\bar{u}$, and $\bar{p}$ are the restrictions of $\mathcal{Y}_1(y_0)$, $\mathcal{U}_1(y_0)$, and $\mathcal{P}_1(y_0)$ to $(0,T)$.
\end{proposition}

The proof is very similar to the one of Proposition \ref{prop:estimate_one_step}. Basically, one needs to replace $D\mathcal{V}_k$ by $Q$ everywhere in the proof.
The last component of $\Phi_2$ must be replaced by $p(T)-Qy(T)$. The variable $q$ which is introduced later must be redefined as follows:
$q= Qy - D \mathcal{V}(y)$.
Then, we have
\begin{equation*}
\| q \|_Y= \| Qy- D \mathcal{V}(y) \|_Y
\leq \| (Q-\Pi) y \|_Y + \| \Pi y - D \mathcal{V}(y) \|_Y
\leq \| Q - \Pi \|_{\mathcal{L}(Y)} \| y \|_Y + M \| y \|_Y^2
\end{equation*}
and the proposition follows.

The statement of Proposition \ref{prop:other_estim_one_step} is unchanged. In Corollary \ref{corollary}, estimates $(d)$ and $(e)$ write now:
\begin{align*}
\max( \| y_{T,k} - \bar{y} \|_{W(0,\tau)} , \| u_{T,k} - \bar{u} \|_{L^2(0,T)} ) \leq \ & M e^{-\lambda (T- \tau)- \lambda T} \big( \| Q - \Pi \|_{\mathcal{L}(Y)} + e^{-\lambda T} \| y_0 \|_Y \big) \| y_0 \|_Y \\
\| y_{T,k}(t)-\bar{y}(\tau) \|_{Y} \leq \ & M e^{-\lambda (T- \tau)- \lambda T} \big( \| Q - \Pi \|_{\mathcal{L}(Y)} + e^{-\lambda T} \| y_0 \|_Y \big) \| y_0 \|_Y.
\end{align*}

We finally obtain the following theorem.

\begin{theorem} \label{thm:main_bis}
There exist $\tau_0 > 0$,  $\delta > 0$, and $M > 0$ such that for all $\tau \geq \tau_0$, for all $T \geq \tau$, for all $Q \in \mathcal{Q}$, and for all $y_0 \in B_Y(\delta)$, the Receding-Horizon method with quadratic penalty cost is well-posed. Moreover, the following estimates hold true:
\begin{align}
\max( \| y_{RH} - \bar{y} \|_{W(0,\infty)}, \| u_{RH} - \bar{u} \|_{L^2(0,\infty)} )
\leq \ & M e^{-\lambda(T-\tau)- \lambda T} \big( \| Q-\Pi \|_{\mathcal{L}(Y)} + e^{-\lambda T} \| y_0 \|_Y \big) \| y_0 \|_Y \label{eq:final_estim1_bis} \\
\mathcal{J}(u_{RH},y_0) - \mathcal{V}(y_0) \leq \ & M e^{-2\lambda(T-\tau)- 2\lambda T}  \big( \| Q-\Pi \|_{\mathcal{L}(Y)} + e^{-\lambda T} \| y_0 \|_Y \big)^2 \| y_0 \|_Y^2, \label{eq:final_estim2_bis}
\end{align}
where $\bar{u}$ is the unique solution to problem \eqref{eqProblem} and $\bar{y}$ the associated trajectory.
\end{theorem}

\begin{remark}
The same comment as in Remark \ref{rem:th1} regarding the dependence of \eqref{eq:final_estim1_bis} with respect to $\tau$ and $T$ can be made. For $Q= \Pi$, estimates \eqref{eq:final_estim1_bis} and \eqref{eq:final_estim2_bis} coincide with \eqref{eq:final_estim1} and \eqref{eq:final_estim2}, respectively, for $k=2$.
\end{remark}

\section{Numerical illustration}

This section is dedicated to the numerical illustration of estimates \eqref{eq:final_estim1} and \eqref{eq:final_estim1_bis}. We focus on the dependence of $\| u_{RH}-\bar{u} \|_{L^2(0,\infty)}$ with respect to the sampling time $\tau$ and the prediction horizon $T$. We consider for this purpose a stabilization problem with state variable of dimension 2, described by the following data:
\begin{equation*}
A= \begin{pmatrix} 0.5 & 1 \\ 0 & -1 \end{pmatrix}, \quad
B= \begin{pmatrix} 1 \\ 1 \end{pmatrix}, \quad
N= \begin{pmatrix} -0.2 & -0.2 \\ 0 & -0.2 \end{pmatrix}, \quad
C= \begin{pmatrix} 1 & 0 \\ 0 & 1 \end{pmatrix}, \quad
y_0= \begin{pmatrix} 1 \\ 1 \end{pmatrix}, \quad
\alpha= 0.1.
\end{equation*}
We have generated different controls with the RHC algorithm, for values of $\tau$ and $T$ ranging from $0.1$ to $2.8$ and for the following three terminal cost functions: $\phi= 0$ (case $k=1$), $\phi= \mathcal{V}_2$ (case $k=2$), and $\phi= \mathcal{V}_3$ (case $k=3$).
All optimal control problems have been solved with the limited-memory BFGS method, with a tolerance of $10^{-12}$ for the $L^2$-norm of the gradient of the reduced cost function. For the discretization of the state equation, we have used the Runge-Kutta method of order 4 with time-step equal to $0.01$. The approximations of the optimal control are computed on the interval $(0,5)$.

As a consequence of estimates \eqref{eq:final_estim1} and \eqref{eq:final_estim1_bis}, there exist for each of the three different cost functions two constants $\tau_0>0$ and $M>0$, both independent of $\tau$ and $T$, such that $\| u_{RH} - \bar{u} \|_{L^2(0,\infty)} \leq M e^{-(k+1) \lambda T + \lambda \tau}$, for $\tau_0 \leq \tau \leq T$.
Thus the quantity
\begin{equation*}
\rho(\tau,T):= \ln( \| u_{RH} - \bar{u} \|_{L^2(0,\infty)} ) + (k+1) \lambda T - \lambda \tau
\end{equation*}
is bounded from above, for sufficiently large values of $\tau$.
The results obtained for $\| u_{RH}-\bar{u} \|_{L^2(0,\infty)}$ and $\rho(\tau,T)$ are shown on Figures \ref{fig:dist} and \ref{fig:rho}, where $\lambda \approx 1.5$.

\begin{figure}[p!]
\begin{center}
Case $k=1$: $\phi= 0$.
\end{center}
\vspace{2mm}
\begin{equation*}
{\scriptstyle
\begin{array}{|c||c|c|c|c|c|c|c|c|c|c|}
\hline
& \multicolumn{10}{c|}{T} \\ \hline
\tau & 0.1 & 0.4 & 0.7 & 1.0 & 1.3 & 1.6 & 1.9 & 2.2 & 2.5 & 2.8 \\ \hline
0.1 & 4.3 \,\mathrm{e}{+0} & 8.3 \,\mathrm{e}{-1} & 2.6 \,\mathrm{e}{-1} & 1.1 \,\mathrm{e}{-1} & 4.7 \,\mathrm{e}{-2} & 2.0 \,\mathrm{e}{-2} & 8.1 \,\mathrm{e}{-3} & 3.3 \,\mathrm{e}{-3} & 1.4 \,\mathrm{e}{-3} & 5.5 \,\mathrm{e}{-4} \\
0.4 & & 1.6 \,\mathrm{e}{+0} & 3.9 \,\mathrm{e}{-1} & 1.6 \,\mathrm{e}{-1} & 6.9 \,\mathrm{e}{-2} & 2.9 \,\mathrm{e}{-2} & 1.2 \,\mathrm{e}{-2} & 4.9 \,\mathrm{e}{-3} & 2.0 \,\mathrm{e}{-3} & 8.2 \,\mathrm{e}{-4} \\
0.7 & & & 5.8 \,\mathrm{e}{-1} & 2.2 \,\mathrm{e}{-1} & 9.7 \,\mathrm{e}{-2} & 4.2 \,\mathrm{e}{-2} & 1.7 \,\mathrm{e}{-2} & 7.2 \,\mathrm{e}{-3} & 2.9 \,\mathrm{e}{-3} & 1.2 \,\mathrm{e}{-3} \\
1.0 & & & & 2.7 \,\mathrm{e}{-1} & 1.3 \,\mathrm{e}{-1} & 6.0 \,\mathrm{e}{-2} & 2.6 \,\mathrm{e}{-2} & 1.1 \,\mathrm{e}{-2} & 4.5 \,\mathrm{e}{-3} & 1.8 \,\mathrm{e}{-3} \\
1.3 & & & & & 1.5 \,\mathrm{e}{-1} & 8.2 \,\mathrm{e}{-2} & 3.8 \,\mathrm{e}{-2} & 1.7 \,\mathrm{e}{-2} & 6.9 \,\mathrm{e}{-3} & 2.8 \,\mathrm{e}{-3} \\
1.6 & & & & & & 8.6 \,\mathrm{e}{-2} & 5.2 \,\mathrm{e}{-2} & 2.5 \,\mathrm{e}{-2} & 1.1 \,\mathrm{e}{-2} & 4.4 \,\mathrm{e}{-3} \\
1.9 & & & & & & & 5.3 \,\mathrm{e}{-2} & 3.3 \,\mathrm{e}{-2} & 1.6 \,\mathrm{e}{-2} & 6.8 \,\mathrm{e}{-3} \\
2.2 & & & & & & & & 3.4 \,\mathrm{e}{-2} & 2.1 \,\mathrm{e}{-2} & 1.0 \,\mathrm{e}{-2} \\
2.5 & & & & & & & & & 2.1 \,\mathrm{e}{-2} & 1.3 \,\mathrm{e}{-2} \\
2.8 & & & & & & & & & & 1.4 \,\mathrm{e}{-2} \\ \hline
\end{array}
}
\end{equation*}
\vspace{6mm}
\begin{center}
Case $k=2$: $\phi(y)= \mathcal{V}_2(y)= \frac{1}{2} \langle y, \Pi y \rangle_Y$.
\end{center}
\vspace{2mm}
\begin{equation*}
{\scriptstyle
\begin{array}{|c||c|c|c|c|c|c|c|c|c|c|}
\hline
& \multicolumn{10}{c|}{T} \\ \hline
\tau & 0.1 & 0.4 & 0.7 & 1.0 & 1.3 & 1.6 & 1.9 & 2.2 & 2.5 & 2.8 \\ \hline
0.1 & 6.8 \,\mathrm{e}{-1} & 1.5 \,\mathrm{e}{-2} & 3.8 \,\mathrm{e}{-4} & 1.8 \,\mathrm{e}{-4} & 8.0 \,\mathrm{e}{-5} & 2.5 \,\mathrm{e}{-5} & 7.2 \,\mathrm{e}{-6} & 2.0 \,\mathrm{e}{-6} & 5.2 \,\mathrm{e}{-7} & 1.4 \,\mathrm{e}{-7} \\
0.4 & & 4.9 \,\mathrm{e}{-2} & 1.8 \,\mathrm{e}{-3} & 1.8 \,\mathrm{e}{-4} & 1.1 \,\mathrm{e}{-4} & 3.9 \,\mathrm{e}{-5} & 1.1 \,\mathrm{e}{-5} & 3.1 \,\mathrm{e}{-6} & 8.4 \,\mathrm{e}{-7} & 2.2 \,\mathrm{e}{-7} \\
0.7 & & & 7.2 \,\mathrm{e}{-3} & 5.9 \,\mathrm{e}{-4} & 1.2 \,\mathrm{e}{-4} & 5.0 \,\mathrm{e}{-5} & 1.6 \,\mathrm{e}{-5} & 4.5 \,\mathrm{e}{-6} & 1.2 \,\mathrm{e}{-6} & 3.2 \,\mathrm{e}{-7} \\
1.0 & & & & 3.2 \,\mathrm{e}{-3} & 2.6 \,\mathrm{e}{-4} & 5.0 \,\mathrm{e}{-5} & 2.0 \,\mathrm{e}{-5} & 6.3 \,\mathrm{e}{-6} & 1.8 \,\mathrm{e}{-6} & 4.8 \,\mathrm{e}{-7} \\
1.3 & & & & & 1.5 \,\mathrm{e}{-3} & 1.1 \,\mathrm{e}{-4} & 2.0 \,\mathrm{e}{-5} & 8.3 \,\mathrm{e}{-6} & 2.6 \,\mathrm{e}{-6} & 7.2 \,\mathrm{e}{-7} \\
1.6 & & & & & & 6.5 \,\mathrm{e}{-4} & 4.6 \,\mathrm{e}{-5} & 8.1 \,\mathrm{e}{-6} & 3.4 \,\mathrm{e}{-6} & 1.0 \,\mathrm{e}{-6} \\
1.9 & & & & & & & 2.8 \,\mathrm{e}{-4} & 1.9 \,\mathrm{e}{-5} & 3.3 \,\mathrm{e}{-6} & 1.4 \,\mathrm{e}{-6} \\
2.2 & & & & & & & & 1.1 \,\mathrm{e}{-4} & 7.9 \,\mathrm{e}{-6} & 1.4 \,\mathrm{e}{-6} \\
2.5 & & & & & & & & & 4.7 \,\mathrm{e}{-5} & 3.2 \,\mathrm{e}{-6} \\
2.8 & & & & & & & & & & 1.9 \,\mathrm{e}{-5} \\ \hline
\end{array}
}
\end{equation*}
\vspace{6mm}
\begin{center}
Case $k=3$: $\phi(y)= \mathcal{V}_3(y)$.
\end{center}
\vspace{2mm}
\begin{equation*}
{\scriptstyle
\begin{array}{|c||c|c|c|c|c|c|c|c|c|c|}
\hline
& \multicolumn{10}{c|}{T} \\ \hline
\tau & 0.1 & 0.4 & 0.7 & 1.0 & 1.3 & 1.6 & 1.9 & 2.2 & 2.5 & 2.8 \\ \hline
0.1 & 1.4 \,\mathrm{e}{-1} & 1.7 \,\mathrm{e}{-4} & 1.5 \,\mathrm{e}{-5} & 3.2 \,\mathrm{e}{-6} & 1.2 \,\mathrm{e}{-6} & 2.7 \,\mathrm{e}{-7} & 5.2 \,\mathrm{e}{-8} & 9.3 \,\mathrm{e}{-9} & 1.9 \,\mathrm{e}{-9} & 1.0 \,\mathrm{e}{-9} \\
0.4 & & 1.0 \,\mathrm{e}{-3} & 4.6 \,\mathrm{e}{-5} & 3.3 \,\mathrm{e}{-6} & 1.8 \,\mathrm{e}{-6} & 4.4 \,\mathrm{e}{-7} & 8.7 \,\mathrm{e}{-8} & 1.6 \,\mathrm{e}{-8} & 3.5 \,\mathrm{e}{-9} & 1.4 \,\mathrm{e}{-9} \\
0.7 & & & 1.5 \,\mathrm{e}{-4} & 1.5 \,\mathrm{e}{-5} & 1.8 \,\mathrm{e}{-6} & 5.6 \,\mathrm{e}{-7} & 1.2 \,\mathrm{e}{-7} & 2.2 \,\mathrm{e}{-8} & 4.0 \,\mathrm{e}{-9} & 2.2 \,\mathrm{e}{-9} \\
1.0 & & & & 7.7 \,\mathrm{e}{-5} & 5.1 \,\mathrm{e}{-6} & 5.3 \,\mathrm{e}{-7} & 1.5 \,\mathrm{e}{-7} & 3.1 \,\mathrm{e}{-8} & 5.7 \,\mathrm{e}{-9} & 2.5 \,\mathrm{e}{-9} \\
1.3 & & & & & 2.8 \,\mathrm{e}{-5} & 1.5 \,\mathrm{e}{-6} & 1.4 \,\mathrm{e}{-7} & 3.9 \,\mathrm{e}{-8} & 8.7 \,\mathrm{e}{-9} & 2.8 \,\mathrm{e}{-9} \\
1.6 & & & & & & 8.4 \,\mathrm{e}{-6} & 4.1 \,\mathrm{e}{-7} & 3.6 \,\mathrm{e}{-8} & 1.1 \,\mathrm{e}{-8} & 2.8 \,\mathrm{e}{-9} \\
1.9 & & & & & & & 2.3 \,\mathrm{e}{-6} & 1.1 \,\mathrm{e}{-7} & 9.4 \,\mathrm{e}{-9} & 4.6 \,\mathrm{e}{-9} \\
2.2 & & & & & & & & 6.3 \,\mathrm{e}{-7} & 3.0 \,\mathrm{e}{-8} & 2.6 \,\mathrm{e}{-9} \\
2.5 & & & & & & & & & 1.7 \,\mathrm{e}{-7} & 7.7 \,\mathrm{e}{-9} \\
2.8 & & & & & & & & & & 4.3 \,\mathrm{e}{-8} \\ \hline
\end{array}
}
\end{equation*}
\caption{$\| u_{RH} - \bar{u} \|_{L^2(0,\infty)}$, for various values of $\tau$ and $T$ and for $k=1,2,3$.}
\label{fig:dist}
\end{figure}

\begin{figure}[p!]
\begin{center}
Case $k=1$: $\phi= 0$.
\end{center}
\vspace{2mm}
\begin{equation*}
{\scriptstyle
\begin{array}{|c||c|c|c|c|c|c|c|c|c|c|}
\hline
& \multicolumn{10}{c|}{T} \\ \hline
\tau & 0.1 & 0.4 & 0.7 & 1.0 & 1.3 & 1.6 & 1.9 & 2.2 & 2.5 & 2.8 \\ \hline
0.1 & 1.6 & 0.9 & 0.6 & 0.7 & 0.7 & 0.7 & 0.7 & 0.8 & 0.8 & 0.8 \\
0.4 & & 1.1 & 0.6 & 0.6 & 0.6 & 0.7 & 0.7 & 0.7 & 0.7 & 0.7 \\
0.7 & & & 0.5 & 0.4 & 0.5 & 0.6 & 0.6 & 0.6 & 0.6 & 0.6 \\
1.0 & & & & 0.2 & 0.4 & 0.5 & 0.6 & 0.6 & 0.6 & 0.6 \\
1.3 & & & & & 0.0 & 0.3 & 0.5 & 0.6 & 0.6 & 0.6 \\
1.6 & & & & & & -0.1 & 0.3 & 0.5 & 0.6 & 0.6 \\
1.9 & & & & & & & -0.1 & 0.3 & 0.5 & 0.6 \\
2.2 & & & & & & & & -0.1 & 0.3 & 0.5 \\
2.5 & & & & & & & & & -0.1 & 0.3 \\
2.8 & & & & & & & & & & -0.1 \\ \hline
\end{array}
}
\end{equation*}
\vspace{6mm}
\begin{center}
Case $k=2$: $\phi(y)= \mathcal{V}_2(y)= \frac{1}{2} \langle y, \Pi y \rangle_Y$.
\end{center}
\vspace{2mm}
\begin{equation*}
{\scriptstyle
\begin{array}{|c||c|c|c|c|c|c|c|c|c|c|}
\hline
& \multicolumn{10}{c|}{T} \\ \hline
\tau & 0.1 & 0.4 & 0.7 & 1.0 & 1.3 & 1.6 & 1.9 & 2.2 & 2.5 & 2.8 \\ \hline
0.1 & -0.1 & -2.6 & -4.9 & -4.3 & -3.7 & -3.5 & -3.4 & -3.4 & -3.4 & -3.3 \\
0.4 & & -1.8 & -3.8 & -4.7 & -3.8 & -3.6 & -3.4 & -3.4 & -3.3 & -3.3 \\
0.7 & & & -2.8 & -4.0 & -4.2 & -3.7 & -3.6 & -3.5 & -3.4 & -3.4 \\
1.0 & & & & -2.8 & -3.9 & -4.2 & -3.8 & -3.6 & -3.5 & -3.5 \\
1.3 & & & & & -2.6 & -3.9 & -4.2 & -3.8 & -3.5 & -3.5 \\
1.6 & & & & & & -2.5 & -3.8 & -4.2 & -3.7 & -3.5 \\
1.9 & & & & & & & -2.5 & -3.8 & -4.2 & -3.7 \\
2.2 & & & & & & & & -2.5 & -3.8 & -4.2 \\
2.5 & & & & & & & & & -2.5 & -3.8 \\
2.8 & & & & & & & & & & -2.5 \\ \hline
\end{array}
}
\end{equation*}
\vspace{6mm}
\begin{center}
Case $k=3$: $\phi(y)= \mathcal{V}_3(y)$.
\end{center}
\vspace{2mm}
\begin{equation*}
{\scriptstyle
\begin{array}{|c||c|c|c|c|c|c|c|c|c|c|}
\hline
& \multicolumn{10}{c|}{T} \\ \hline
\tau & 0.1 & 0.4 & 0.7 & 1.0 & 1.3 & 1.6 & 1.9 & 2.2 & 2.5 & 2.8 \\ \hline
0.1 & -1.6 & -6.4 & -7.0 & -6.8 & -6.0 & -5.7 & -5.5 & -5.4 & -5.2 & -4.1 \\
0.4 & & -5.1 & -6.4 & -7.2 & -6.0 & -5.6 & -5.5 & -5.4 & -5.1 & -4.1 \\
0.7 & & & -5.6 & -6.2 & -6.5 & -5.9 & -5.6 & -5.5 & -5.4 & -4.2 \\
1.0 & & & & -5.0 & -5.9 & -6.3 & -5.8 & -5.6 & -5.5 & -4.5 \\
1.3 & & & & & -4.6 & -5.8 & -6.3 & -5.8 & -5.5 & -4.8 \\
1.6 & & & & & & -4.5 & -5.7 & -6.3 & -5.8 & -5.3 \\
1.9 & & & & & & & -4.4 & -5.7 & -6.3 & -5.3 \\
2.2 & & & & & & & & -4.4 & -5.6 & -6.3 \\
2.5 & & & & & & & & & -4.4 & -5.6 \\
2.8 & & & & & & & & & & -4.4 \\ \hline
\end{array}
}
\end{equation*}
\caption{ $\rho(\tau,T):= \ln( \| u_{RH} - \bar{u} \|_{L^2(0,\infty} ) + (k+1) \lambda T - \lambda \tau$, for various values of $\tau$ and $T$ and for $k=1,2,3$.}
\label{fig:rho}
\end{figure}

A first observation is that $\| u_{RH}-\bar{u} \|_{L^2(0,\infty)}$ is decreasing with respect to $T$ and increasing with respect to $\tau$. It is also decreasing with respect to $k$, which shows (at least on this particular example) the interest of considering a high-order Taylor expansion of the value function as terminal cost.

Let us examine now the number $\rho$. In order to justify that $\rho$ is constant, we compare the variation of $\rho$ with the variation of $-(k+1) \lambda T + \lambda \tau$ over the considered values of $\tau$ and $T$. We exclude, in the three cases, the results obtained for $\tau= 0.1$, which is acceptable since our estimate only holds for sufficiently large values of $\tau$.
In the first case ($k=1$), the number $\rho(\tau,T)$ takes values between $-0.1$ and $1.1$. The variation of $\rho$ (equal to $1.2$) is rather small in comparison with the variation of the quantity $-2\lambda T + \lambda \tau$, which reaches its maximum, $-0.6$, at $(\tau,T)= (0.4,0.4)$ and its minimum, $-7.8$, at $(\tau,T)=(0.4,2.8)$ (we exclude again the case $\tau= 0.1$).
In the second case $(k=2)$, the number $\rho(\tau,T)$ takes values between $-4.2$ and $-1.8$. The variation of $\rho$ (equal to $2.4$) is small in comparison with the variation of $-3 \lambda T + \lambda \tau$ (equal to $10.8$).
In the third case $(k=3)$, the number $\rho(\tau,T)$ takes values between $-7.2$ and $-4.1$. The variation of $\rho$ (equal to $3.1$) is small in comparison with the variation of $-4 \lambda T + \lambda \tau$ (equal to $14.4$). 
We can therefore consider that the variation of $\rho$ is small in these three cases, and thus that $\rho$ is constant. We finally conclude that our error estimate gives an accurate description of the dependence of $\| u_{RH}- \bar{u} \|_{L^2(0,\infty)}$ with respect to $\tau$ and $T$.

\section{Conclusion}

We have analyzed the RHC algorithm for a class of non-linear stabilization problems. Different types of terminal cost functions have been considered for the sequence of finite-horizon problems to be solved at each iteration. An exponential rate of convergence with respect to the prediction horizon $T$ has been obtained and observed numerically on a simple example.

Future research will focus on the adaptation of our results for other types of non-linearities. As was mentioned in the introduction, our results can be extended to the case of finite-dimensional systems of the form $\dot{y}= Ay + Bu + f(y,u)$ where $f$ and its derivative vanish at 0.
The general case of time-dependent systems of the form $\dot{y}(t)= A(t)y(t) + B(t)u(t) + f(t,y(t),u(t))$ is open. The case where $A$ and $B$ are periodic could be considered, by utilizing the stabilizability results obtained in \cite{WanX16} for infinite-dimensional periodic linear control systems.
Another direction of research is the analysis of the RHC method for problems satisfying the turnpike property. Let us mention that some results have already been obtained in \cite{BreP18} for time-independent linear-quadratic problems, for which the turnpike property holds. Finally, one could generalize our error estimates by taking into account the time-discretization of the finite-horizon problems. It has been shown recently in \cite{GSS18} that non-uniform time-grids are well-suited for solving linear-quadratic optimal control problems with RHC schemes (in a nutshell: a fine grid is used on $(0,\tau)$ and a coarser one on $(\tau,T)$). This result can certainly be extended to a non-linear setting, using the techniques of the present work.

\appendix

\section{Inverse mapping theorem}

For completeness, below we give a formulation of the inverse mapping theorem, used several times in this article.

Let $\Lambda$ be a vector space equipped with two norms, $\| \cdot \|_{\Lambda_a}$ and $\| \cdot  \|_{\Lambda_b}$. The space $\Lambda$, equipped with $\| \cdot \|_{\Lambda_a}$ (resp.\@ $\| \cdot \|_{\Lambda_b}$) is denoted $\Lambda_a$ (resp.\@ $\Lambda_b$). Similarly, let $\Upsilon$ be a vector space equipped with two norms, $\| \cdot \|_{\Upsilon_a}$ and $\| \cdot \|_{\Upsilon_b}$. With the same convention as before, we write $\Upsilon_a$ and $\Upsilon_b$. It is assumed that the spaces $\Lambda_a$, $\Lambda_b$, $\Upsilon_a$, and $\Upsilon_b$ are Banach spaces.

We consider a mapping $\phi\colon \Lambda \rightarrow \Upsilon$, such that $\phi(0)= 0$.

\begin{theorem} \label{thmInvThm}
Assume that $\phi$ is continuously differentiable from $\Lambda_a$ to $\Upsilon_a$ and from $\Lambda_b$ to $\Upsilon_b$. We assume that $D\phi(0)$, as a linear mapping from $\Lambda_a$ to $\Upsilon_a$ and as a linear mapping from $\Lambda_b$ to $\Upsilon_b$ is bijective with a bounded inverse. Let $M_0 > 0$ be such that
\begin{equation} \label{eqInvThmAss1}
\| D\phi(0)^{-1} \|_{\mathcal{L}(\Lambda_a,\Upsilon_a)} \leq M_0 \quad \text{and} \quad
\| D\phi(0)^{-1} \|_{\mathcal{L}(\Lambda_b,\Upsilon_b)} \leq M_0.
\end{equation}
Assume further that there exist $\delta_0 > 0$ and $M_1 >0$ such that for all $x_1$ and $x_2 \in B_{\Lambda_a}(\delta_0)$,
\begin{equation} \label{eqInvThmAss2}
\begin{cases} \begin{array}{l}
\| D\phi(x_2)-D\phi(x_1) \|_{\mathcal{L}(\Lambda_a,\Upsilon_a)} \leq M_1 \| x_2 - x_1 \|_{\Lambda_a}, \\
\| D\phi(x_2)-D\phi(x_1) \|_{\mathcal{L}(\Lambda_b,\Upsilon_b)} \leq M_1 \| x_2 - x_1 \|_{\Lambda_a}.
\end{array}
\end{cases}
\end{equation}
Let $\delta'>0$ and let $\delta>0$ be such that
$\delta' \leq \delta_0$,
$M_0 M_1 \delta' < 1$,
and $\frac{M_0 \delta}{1-M_0 M_1 \delta'} \leq \delta'$.
Then, there exists a mapping $\mathcal{X} \colon B_{\Upsilon_a}(\delta) \rightarrow B_{\Lambda_a}(\delta')$ such that for all $y \in B_{\Upsilon_a}(\delta)$, $\mathcal{X}(y)$ is the unique solution to
\begin{equation} \label{eqInvThmEquation}
\phi(x)=y \quad \text{and} \quad \| x \|_{\Lambda_a} \leq \delta'.
\end{equation}
Moreover, for all $y_1$ and $y_2 \in B_{\Upsilon_a}(\delta)$,
\begin{equation} \label{eqInvThmLipschitz}
\begin{cases}
\begin{array}{l}
\| \mathcal{X}(y_2)-\mathcal{X}(y_1) \|_{\Lambda_a} \leq M_0(1-M_0 M_1 \delta')^{-1} \| y_2-y_1 \|_{\Upsilon_a}, \\
\| \mathcal{X}(y_2)-\mathcal{X}(y_1) \|_{\Lambda_b} \leq M_0(1-M_0 M_1 \delta')^{-1} \| y_2-y_1 \|_{\Upsilon_b}.
\end{array}
\end{cases}
\end{equation}
\end{theorem}

\begin{remark} \label{remark:19}
\begin{enumerate}
\item For $\| \cdot \|_{\Lambda_a}= \| \cdot \|_{\Lambda_b}$ and $\| \cdot \|_{\Upsilon_a} = \| \cdot \|_{\Upsilon_b}$,  the above theorem is the classical inverse function theorem. The particularity of the formulation of the theorem is that the mapping $x \in \Lambda_a \mapsto D\phi(x) \in \mathcal{L}(\Lambda_b,\Upsilon_b)$ is locally Lipschitz-continuous, see \eqref{eqInvThmAss2}.
\item The constants $\delta'$ and $\delta$ as well as the Lipschitz modulus of $\mathcal{X}$ can be explicitly obtained as functions of the upper bound on $\| D\phi(0)^{-1} \|_{\mathcal{L}(Y,X)}$ and of the Lipschitz modulus of $D\phi(\cdot)$. In Proposition \ref{prop:estimate_one_step} and Proposition \ref{prop:other_estim_one_step} they  can be both chosen independently of $T$.
\end{enumerate}
\end{remark}

\begin{proof}[Proof of Theorem \ref{thmInvThm}]
\emph{Step 1}: Existence of a solution to \eqref{eqInvThmEquation}.\\
Fix $y \in B_{\Upsilon_a}(\delta)$. Consider the sequence $(x_n)_{n \in \mathbb{N}}$ in $X$, defined as follows:
\begin{equation} \label{eqSequenceInvFct}
x_0= 0, \quad
x_{n+1}= x_n + D\phi(0)^{-1} (y-\phi(x_n)), \quad \forall n \in \mathbb{N}.
\end{equation}
Let us prove by induction that for all $n \in \mathbb{N} \backslash \{ 0 \}$,
\begin{equation} \label{eqInductionInvFct}
\| x_n \|_{\Lambda_a} \leq \frac{1-(M_0 M_1 \delta')^n}{1- M_0 M_1 \delta'} M_0 \delta \quad \text{and} \quad
\| x_n - x_{n-1} \|_{\Lambda_a} \leq (M_0 M_1 \delta')^{n-1} M_0 \delta.
\end{equation}
Note that for all $n \in \mathbb{N}$,
\begin{equation*}
\frac{1-(M_0 M_1 \delta')^n}{1- M_0 M_1 \delta'} M_0 \delta
\leq \frac{M_0 \delta}{1-M_0 M_1 \delta'} \leq \delta'.
\end{equation*}
By \eqref{eqInvThmAss1}, we have
\begin{equation*}
\| x_1 \|_{\Lambda_a} = \| x_1- x_0 \|_{\Lambda_a} = \| D\phi(0)^{-1} y \|_{\Lambda_a}
\leq \| D \phi(0)^{-1} \|_{\mathcal{L}(\Upsilon_a,\Lambda_a)} \| y \|_{\Upsilon_a}
\leq M_0 \delta.
\end{equation*}
Therefore, the assertion holds true for $n=1$. Assume that it holds up to some $n \in \mathbb{N} \backslash \{ 0 \}$. We have
\begin{equation*}
\phi(x_n)- \phi(x_{n-1})= \int^1_0 D\phi(\theta x_n + (1-\theta) x_{n-1})(x_{n}-x_{n-1})\,d\theta,
\end{equation*}
where $\| \theta x_n + (1-\theta) x_{n-1} \|_{\Lambda_a} \leq \theta \| x_n \|_{\Lambda_a} + (1-\theta) \| x_{n-1} \|_{\Lambda_a} \leq \delta' \leq \delta_0$.
By construction, we have
\begin{align*}
x_{n+1}- x_n = \ & D\phi(0)^{-1} \big( y- \phi(x_n) \big) \\
= \ & D\phi(0)^{-1} \big( y-\phi(x_{n-1}) + \phi(x_{n-1})- \phi(x_n) \big) \\
= \ & D\phi(0)^{-1} \big(\int^1_0 (D\phi(0) - D\phi(\theta x_n +(1-\theta) x_{n-1})\,)\,(x_n-x_{n-1}) \, d\theta \big).
\end{align*}
Using \eqref{eqInvThmAss1} and \eqref{eqInvThmAss2}, we deduce that
\begin{align*}
\| x_{n+1} - x_n \|_{\Lambda_a}
\leq M_0 M_1 \delta' \| x_n - x_{n-1} \|_{\Lambda_a}
\leq (M_0 M_1 \delta')^n M_0 \delta.
\end{align*}
Moreover,
\begin{align*}
\| x_{n+1} \|_{\Lambda_a}
\leq & \ \| x_{n+1} - x_n \|_{\Lambda_a} + \| x_n \|_{\Lambda_a} \\
\leq & \ (M_0 M_1 \delta')^n M_0 \delta + \frac{1-(M_0 M_1 \delta')^n}{1- M_0 M_1 \delta'} M_0 \delta
=  \frac{1-(M_0 M_1 \delta)^{n+1}}{1- M_0 M_1 \delta} M_0 \delta',
\end{align*}
and thus the assertion is true for $n+1$.

As a consequence of \eqref{eqInductionInvFct}, the sequence $(x_n)_{n \in \mathbb{N}}$ is a Cauchy sequence and thus possesses a limit, say $x$, such that $\| x \|_{\Lambda_a} \leq \delta'$. Passing to the limit in \eqref{eqSequenceInvFct}, we obtain that $\phi(x)= y$.

\emph{Step 2}: Uniqueness of the solution to \eqref{eqInvThmEquation}.\\
Let $x' \in B_{\Lambda_a}(\delta')$ be such that $\phi(x')= y$.
Since 
\begin{equation*}
0= \phi(x')-\phi(x)= \int_0^1 D \phi(\theta x' + (1-\theta) x)(x'-x) \dd \theta,
\end{equation*}
we have
\begin{equation*}
x'-x= D\phi(0)^{-1} \Big( \int_0^1 (D\phi(0)-D \phi(\theta x' + (1-\theta) x))(x'-x) \dd \theta \Big).
\end{equation*}
Therefore, by \eqref{eqInvThmAss1} and \eqref{eqInvThmAss2},
$\| x'-x \|_{\Lambda_a} \leq M_0 M_1 \delta \| x'-x \|_{\Lambda_a}$.
Since $M_0 M_1 \delta' < 1$, we deduce that $x= x'$.

\emph{Step 3}: Lipschitz-continuity of the mapping $\mathcal{X}$.\\
Let $y$ and $y' \in B_{\Upsilon_a}(\delta)$, and let $x$ and $x' \in B_{\Lambda_a}(\delta')$ be such that $\phi(x)=y$ and $\phi(x')= y'$. Since
\begin{equation*}
y'-y= \phi(x')-\phi(x)= \int^1_0 D\phi(\theta x' + (1-\theta) x)(x'-x)\,d\theta,
\end{equation*}
we have
\begin{equation*}
x'-x= D\phi(0)^{-1} \big( (y'-y) + (D\phi(0)- \int^1_0 D\phi(\theta x' + (1-\theta) x)(x'-x)\,d\theta\, ) \, \big).
\end{equation*}
Using \eqref{eqInvThmAss1} and \eqref{eqInvThmAss2}, we obtain that
\begin{equation} \label{eqInvThm1}
\| x'-x \|_{\Lambda_a} \leq M_0 \| y'-y \|_{\Upsilon_a} + M_0 M_1 \delta' \| x'-x \|_{\Lambda_a}
\end{equation}
and finally that
\begin{equation} \label{eqInvThm2}
\| x'-x \|_{\Lambda_a} \leq \frac{M_0}{1-M_0 M_1 \delta'} \| y'-y \|_{\Upsilon_a}.
\end{equation}
Estimates \eqref{eqInvThm1} and \eqref{eqInvThm2} are both true when using the norms $\| \cdot \|_{\Lambda_b}$ and $\| \cdot \|_{\Upsilon_b}$, and thus \eqref{eqInvThmLipschitz} is proved.

\end{proof}

\section{Technical comments}


\begin{proof}[Complement of proof, Proposition \ref{prop:estimate_one_step}]
We justify here that $\Phi_2$, considered from $\Lambda_{T,-\lambda}$ to $\Upsilon_{T,-\lambda}$ is differentiable, with a Lipschitz-continuous derivative on bounded subsets of $W(0,T) \times L^2(0,T) \times W(0,T)$.
To this purpose, we focus on the following mapping:
\begin{equation*}
\varphi \colon (y,u) \in W_{-\lambda}(0,T) \times L_{-\lambda}^2(0,T)
\mapsto Nyu \in L_{-\lambda}^2(0,T;V^*).
\end{equation*}
Note first that
\begin{align*}
\| \varphi(y,u) \|_{L_{-\lambda}^2(0,T;V^*)}
= \ & \| e^{-\lambda \cdot} Nyu \|_{L^2(0,T;V^*)} \\
\leq \ & e^{\lambda T} \| e^{-\lambda \cdot} Ny \|_{L^\infty(0,T;V^*)} \| e^{-\lambda \cdot} u \|_{L^2(0,T)} \\
\leq \ & M e^{\lambda T} \| y \|_{W_{-\lambda}(0,T)} \| u \|_{L^2_{-\lambda}(0,T)}.
\end{align*}
Furthermore, we have
\begin{align*}
\varphi(y_2,u_2)-\varphi(y_1,u_1)
= N(y_2-y_1)u_1 + Ny_1(u_2-u_1) + N(y_2-y_1)(u_2-u_1).
\end{align*}
It follows that
\begin{align*}
\| N(y_2-y_1)u_1 \|_{L_{-\lambda}^2(0,T;V^*)}
\leq \ & e^{\lambda T} \| N \|_{\mathcal{L}(Y;V^*)} \| y_2-y_1 \|_{W_{-\lambda}(0,T)} \| u_1 \|_{L_{-\lambda}^2(0,T)}, \\
\| Ny_1(u_2-u_1) \|_{L_{-\lambda}^2(0,T;V^*)}
\leq \ & e^{\lambda T} \| N \|_{\mathcal{L}(Y;V^*)} \| y_1 \|_{W_{-\lambda}(0,T)} \| u_2-u_1 \|_{L_{-\lambda}^2(0,T)}
\end{align*}
and
\begin{align*}
\| N(y_2-y_1)(u_2-u_1) \|_{L_{-\lambda}^2(0,T;V^*)}
\leq \ & e^{\lambda T} \| N \|_{\mathcal{L}(Y;V^*)} \| y_2-y_1 \|_{W_{-\lambda}(0,T)} \| u_2 - u_1 \|_{L_{-\lambda}^2(0,T)} \\
\leq \ & \frac{1}{2} e^{\lambda T} \| N \|_{\mathcal{L}(Y;V^*)} \big( \| y_2-y_1 \|_{W_{-\lambda}(0,T)}^2 + \| u_2 - u_1 \|_{L_{-\lambda}^2(0,T)}^2 \big).
\end{align*}
This justifies that $\varphi$ is differentiable, with
$D\varphi(y,u)(z,v)= Nyv + Nuz$.
Finally, we have
\begin{align*}
& \| D\varphi(y_2,u_2)(z,v)-D\varphi(y_1,u_1)(z,v) \|_{L_{-\lambda}^2(0,T;V^*)} \\
& \qquad \leq \| N(y_2-y_1)(e^{-\lambda \cdot} v) \|_{L^2(0,T;V^*)} +
\| N(e^{-\lambda \cdot} z)(u_2-u_1) \|_{L^2(0,T;V^*)} \\
& \qquad \leq  M \big( \| y_2- y_1 \|_{W(0,T)} + \| u_2- u_1 \|_{L^2(0,T)} \big) \big( \| z \|_{W_{-\lambda}(0,T)} + \| v \|_{L_{-\lambda}^2(0,T)} \big),
\end{align*}
thus,
\begin{align*}
& \| D\varphi(y_2,u_2)-D\varphi(y_1,u_1) \|_{\mathcal{L}(W_{-\lambda}(0,T) \times L_{-\lambda}^2(0,T);L_{-\lambda}^2(0,T;V^*))} \\
& \qquad \leq M \big( \| y_2- y_1 \|_{W(0,T)} + \| u_2- u_1 \|_{L^2(0,T)} \big).
\end{align*}
as was to be proved. We emphasize that the constant $M$ in the above inequality does not depend on $T$. The other terms can be treated similarly, in order to prove \eqref{eq_strange_lip_cont}.
\end{proof}

\begin{remark} \label{remark:why_new_ift}
We can observe that the mapping
\begin{equation*}
(y,u) \in W_{-\lambda}(0,T) \times L_{-\lambda}^2(0,T)
\mapsto D \varphi \in \mathcal{L}(W_{-\lambda}(0,T) \times L_{-\lambda}^2(0,T);L_{-\lambda}^2(0,T;V^*))
\end{equation*}
is globally Lipschitz-continuous:
\begin{align*}
& \| D\varphi(y_2,u_2)-D\varphi(y_1,u_1) \|_{\mathcal{L}(W_{-\lambda}(0,T) \times L_{-\lambda}^2(0,T);L_{-\lambda}^2(0,T;V^*))} \\
& \qquad \leq M e^{\lambda T} \big( \| y_2- y_1 \|_{W_{-\lambda}(0,T)} + \| u_2- u_1 \|_{L_{-\lambda}^2(0,T)} \big).
\end{align*}
The modulus, however, grows with $T$. This is the reason why the implicit function theorem cannot be applied in a direct way in Proposition \ref{prop:estimate_one_step} with $\Phi_2$ defined from $\Lambda_{T,-\lambda}$ to $\Upsilon_{T,\lambda}$. The formulation of the implicit theorem that we suggest allows to overcome this difficulty and should also be useful when investigating the RHC algorithm for  other stabilization problems.
\end{remark}

\bibliographystyle{plain}

\end{document}